\documentclass[11pt, a4paper]{article}
\usepackage{amsmath, amsfonts}
\usepackage[usenames]{color}
\usepackage[latin1]{inputenc}
\usepackage[table, dvipsnames]{xcolor}
\usepackage{array}
\newcolumntype{P}[1]{>{\centering\arraybackslash}p{#1}}
\newcolumntype{M}[1]{>{\centering\arraybackslash}m{#1}}
\usepackage{multirow}
\usepackage{tabularx}
\usepackage[all]{xy}
\usepackage{tikz-cd}
\usepackage{tikz}
\usepackage{float}
\usepackage{xcolor}
\usepackage{graphicx}
\usepackage{geometry}
\usepackage{mathtools}
\usepackage{mathrsfs}
\usepackage{tcolorbox}
\usepackage[table]{xcolor}
\usepackage{hyperref} %transformas o sumario e as citacoes em hiperlinks
\hypersetup{bookmarks = true, % mostra a barra de bookmarks
	colorlinks = true, % false: boxed links; true: colored links
	linkcolor = blue, % cor para links
	citecolor = blue, % cor para links em bibliografia
}
\newcommand\blfootnote[1]{
    \begingroup
    \renewcommand\thefootnote{}\footnote{#1}
    \addtocounter{footnote}{-1}
    \endgroup
}
\usepackage{tcolorbox}
\definecolor{mygray}{rgb}{0.9,0.9,0.9}
\newtcolorbox{mybox}{
	arc=0pt,
	boxrule=0pt,
	colback=mygray,
	width=1\textwidth,   % this option controls the width of the box
	colupper=black,
	fontupper=\bfseries
}

\newcommand{\xdownarrow}[1]{%
	{\left\downarrow\vbox to #1{}\right.\kern-\nulldelimiterspace}
}

\newcommand{\xuparrow}[1]{%
	{\left\uparrow\vbox to #1{}\right.\kern-\nulldelimiterspace}
}

\usepackage{graphicx}
\usepackage{amssymb}
\usepackage{amsmath}

\newtheorem{theorem}{Theorem}[section]

\newtheorem{corollary}[theorem]{Corollary}

\newtheorem{definition}[theorem]{Definition}
\newtheorem{example}[theorem]{Example}

\newtheorem{lemma}[theorem]{Lemma}

\newtheorem{proposition}[theorem]{Proposition}
\newtheorem{remark}[theorem]{Remark}

\newenvironment{proof}[1][Proof]{\noindent\textbf{#1.} }{\ \rule{0.5em}{0.5em}}

\title{\textbf{On reflection maps from the n-space to the n+1-space}}
\author{ \ \ \ \\{Gama, M.B. $ \ \ $ and  $ \ \ $ Silva, O.N. $ \ \ $}}
\date{}

\begin{document}
	
	\maketitle
	
	\begin{abstract}
		
		In this work we consider some problems about a reflected graph map germ $f$ from $(\mathbb{C}^n,0)$ to $(\mathbb{C}^{n+1},0)$. A reflected graph map is a particular case of a reflection map, which is defined using an embedding of $\mathbb{C}^n$ in $\mathbb{C}^{p}$ and then applying the action of a reflection group $G$ on $\mathbb{C}^{p}$. In this work, we present a description of the presentation matrix of $f_*{\cal O}_n$ as an ${\cal O}_{n+1}$-module via $f$ in terms of the action of the associated reflection group $G$. We also give a description for a defining equation of the image of $f$ in terms of the action of $G$. Finally, we provide an upper (and also a lower) bound for the multiplicity of the image of $f$ and some applications.
	\end{abstract}

	
	\section{Introduction}
	
	$ \ \ \ \ $ In this work we consider some problems on reflection maps from $\mathbb{C}^n$ to $\mathbb{C}^{n+1}$. Reflection maps have recently emerged in the literature and have proven to be a very interesting subject for singularity theory, with interesting properties and challenging problems to be explored. These kind of maps were introduced by Peñafort-Sanchis in \cite{[5]} where several interesting  problems are considered such as Lê's Conjecture, normal crossings and $\mathcal{A}$-finite determinacy. Reflection maps were also used to produce the first known counterexample of Ruas's conjecture (\cite{[6b]} see also \cite{[10]}). In order to explain what a reflection map is, let us first introduce some ingredients and notation. \blfootnote{2020 Mathematics Subject Classification. Primary MSC 32S05 - MSC 14JS17.}
	
	Let $GL(\mathbb{C}^{p})$ be the group of all invertible linear maps from $\mathbb{C}^{p}$ to $\mathbb{C}^p$ (the general linear group), and $U(\mathbb{C}^p)$ the group of unitary automorphisms of $\mathbb{C}^p$. A \textit{reflection} in $\mathbb{C}^p$ is a linear map $g:\mathbb{C}^p \longrightarrow \mathbb{C}^p$ which is unitary, has finite order (as an element of $GL(\mathbb{C}^p)$) and the set of points fixed by the action of $g$ has dimension $p-1$. A finite subgroup $G$ of $U(\mathbb{C}^p)$ is a (unitary) reflection group if it is generated by reflections. The cyclic group $\mathbb{Z}_d$ and the dihedral ${D}_{2m}$ are typical examples of reflection groups (as long as they are considered with convenient representations in $GL(\mathbb{C}^p)$, see Remark \ref{obs2}). There is a vast theory regarding reflection groups, and we will only explain in this work what we need to get our results. A general reference for the topic is \cite{[2]} (see also \cite{[2b]}) where one can find a detailed description of reflection groups and also the classification of irreducible reflection groups obtained by Shephard-Todd in 1954 (see \cite{[6]}). 
	
	\begin{figure}[H]
		\centering
		\includegraphics[width=0.7\linewidth]{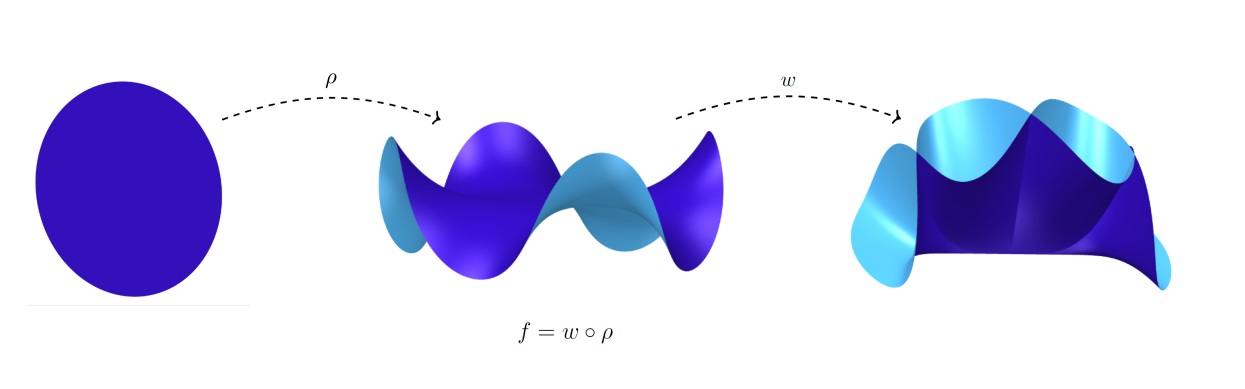}
		\caption{The $C_5$ singularity viewed as a reflection map}
		\label{fig:exe2}
	\end{figure}
	
In few words, a \textit{reflection map} $f:\mathbb{C}^n \longrightarrow \mathbb{C}^p$ is defined by Peñafort-Sanchis in \cite{[5]} simply as the composition of the orbit map  $w:\mathbb{C}^p\longrightarrow \mathbb{C}^p$ of the group $G$ (see \cite[Def. 9.2]{[2]}) with an embedding $\rho:\mathbb{C}^n \hookrightarrow \mathbb{C}^p$, i. e., $f=w\circ \rho$.	
	
	We remark that according to Peñafort-Sanchis, the formulation of a precise definition of a reflection map is attributed to Fernández de Bobadilla.  A reflected graph is a particular case of a reflection map. These kind of maps were also introduced by Peñafort-Sanchis in \cite{[5]} and play a very important role in the theory of reflection maps. For instance, Peñafort-Sanchis shows that with a certain hypothesis any reflection map germ is $\mathcal{A}$-equivalent to a reflected graph map germ (see \cite[Prop. 3.2]{[5]}). Let $h: \mathbb{C}^p \longrightarrow \mathbb{C}^r$ be any holomorphic map and $G$ a reflection group acting on $\mathbb{C}^p$. In few words, we can say that a reflected graph map $(w,h): \mathbb{C}^p\longrightarrow \mathbb{C}^{p+r}$ is the reflection map obtained by taking the embedding of $h$ given by $x \longmapsto(x, h(x))$, and letting $G$ act on $\mathbb {C}^p \times \mathbb{C}^r$, trivially on the second factor. 
	
	The objective of this work is to study some problems about reflected graph map germs from $(\mathbb{C}^n,0)$ to $(\mathbb{C}^{n+1},0)$, where $h$ is a map germ from $(\mathbb{C}^n,0)$ to $(\mathbb{C},0)$. Before we describe the problems we will deal with, let us look at some historical cases.
	
	The simplest class of examples of reflected graph maps are the fold maps. For instance, the $C_5$-singularity of Mond's list \rm(\cite[p.378]{[8]}) defined by $f(x,y)=(x,y^2,xy^3-x^5y)$ (see Figure \ref{fig:exe2}). Historically, fold maps are the first examples of reflection maps that we find in the literature. The first people to study the subject were Bruce in \cite{bruce} and Mond in \cite{[9]}. A fold map germ $f:(\mathbb{C}^2,0) \longrightarrow (\mathbb{C}^3,0)$ is a reflected graph map germ where $\rho: (\mathbb{C}^2,0) \hookrightarrow (\mathbb{C}^3,0)$ is defined by $\rho(x,y)=(x,y,h(x,y))$ and $w: (\mathbb{C}^3,0) \longrightarrow (\mathbb{C}^3,0)$ is the orbit map $w(X,Y,Z)=(X,Y^2,Z)$ of the group $\mathbb{Z}_2 \times G_{id}$, where $G_{id}$ is trivial group.  
	
	Later, in 2008, Marar and Nuño-Ballesteros introduced in \cite{[3]} the double fold maps. These maps are similar to folds maps, but for this class, $w$ is the orbit map of the group $\mathbb{Z}_2 \times \mathbb{Z}_2$ and $f=w \circ \rho$ takes the form $f(x,y)=(x^2,y^2,h(x,y))$. A detailed study of double fold maps is given in \cite{guille}. A typical example of a double fold map germ is the one given by $f(x,y)=(x^2,y^2,x^3+y^3+xy)$ (see Figure \ref{fig:hh}).
	
	\begin{figure}[H]
		\centering
		\includegraphics[width=0.7\linewidth]{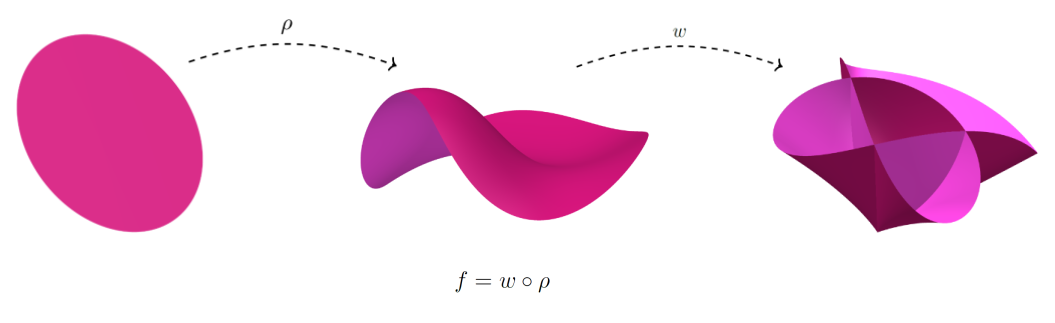}
		\caption{The reflection map $f(x,y)=(x^2,y^2,x^3+y^3+xy)$ (real points).}
		\label{fig:hh}
	\end{figure}
	
	Note that the image of $f$ above is a singular surface in $\mathbb{C}^3$. If $(X,Y,Z)$ denotes a system of coordinates in $\mathbb{C}^3$ then a defining equation of the image of $f$ is 
	\begin{equation}\label{eq7}
		X^2Y^2-2XYZ^2+Z^4-2X^4Y-2XY^4-8X^2Y^2Z-2X^3Z^2-2Y^3Z^2+X^6-2X^3Y^3+Y^6=0.
	\end{equation}
	
	We remark that determining a defining equation of the image of a map germ in general is not an easy task. In this way, let $G$ be a reflection group acting in $\mathbb{C}^n \times \mathbb{C}$, trivially on the second factor, through an orbit map $w=(w_1,\cdots,w_n,Z)$. Let $h:(\mathbb{C}^n,0)\rightarrow (\mathbb{C},0)$ be a holomorphic function. Consider a reflected graph map germ $f:(\mathbb{C}^n,0)\rightarrow (\mathbb{C}^{n+1},0)$, $f(\textit{\textbf{x}})=(w_1(\textit{\textbf{x}}),\cdots,w_n(\textit{\textbf{x}}),h(\textit{\textbf{x}}))$, where $\textit{\textbf{x}}=(x_1,x_2,...,x_n)$. We have that the action of $G$ on $\mathbb{C}^n \times \mathbb{C}$ induces an action on the ring of analytic functions ${\cal O}_{n+1}$ (which is isomorphic to $\mathbb{C}\{X_1,X_2,...,X_n,Z\}$, the ring of convergent series in a neighborhood of the origin). Thus a natural question in this setting is: 
	
	\begin{mybox}
		\textbf{Question 1:}  How can we describe a defining equation of the image of a reflected graph map germ $f:(\mathbb{C}^n,0) \longrightarrow (\mathbb{C}^{n+1},0)$ in terms of the action of $G$ on $h$?
	\end{mybox}
	
	When $n=2$, a defining equation of the image of a reflection map is described in \cite{[1]} in the case of the cyclic groups $\mathbb{Z}_2$, $\mathbb{Z}_3$, $\mathbb{Z}_4$ and the group $\mathbb{Z}_2 \times \mathbb{Z}_2$. In \cite[Ch. 4]{[10]} the case of the groups $\mathbb{Z}_2 \times \mathbb{Z}_3$, $\mathbb{Z}_3 \times \mathbb{Z}_5$ and $\mathbb{Z}_4$ is considered. However, none of these works describe a defining equation of the image in terms of the action of $G$ on $h$.\\

	Also in the case where $n=2$, Marar and Nuño-Ballesteros present a defining equation of the image of $f$ in the case of the group $\mathbb{Z}_2 \times \mathbb{Z}_2$ (see the proof of \cite[Prop. 3.1]{[3]}). The technique they use is to consider the fact that one can see the pushforward $f_{\ast}\mathcal{O}_2$ as an $\mathcal{O}_3$-module, via composition with $f$. In the sequence, using Mond-Pellikaan algorithm (see \cite{[12]})  they determine a presentation matrix of $f_{\ast}\mathcal{O}_2$ and then they find a defining equation of the image as the $0$-Fitting ideal of $f_{\ast}\mathcal{O}_2$, which is simply the determinant of the presentation matrix of $f_{\ast}\mathcal{O}_2$. For instance, a presentation matrix of $f=(x^2,y^2,x^3+y^3+xy)$ (with relation to the basis $1,x,y,xy$) is given by 
	
	$$M= \begin{bmatrix}
		-Z & X & Y & 1 \\
		X^2 & -Z & X & Y \\
		Y^2 & Y & -Z & X \\
		XY & Y^2 & X^2 & -Z \
	\end{bmatrix}.$$
	
	Note that the defining equation of $f=(x^2,y^2,x^3+y^3+xy)$ described in (\ref{eq7}) is precisely the determinant of the matrix $M$ above. Thus a natural question in this setting is: 
	
	\begin{mybox}
		\textbf{Question 2:} Let  $f(\textit{\textbf{x}})=(w(\textit{\textbf{x}}),h(\textit{\textbf{x}}))$ be a reflected graph map germ as in Question $1$. How can we describe the presentation matrix of the pushforward $f_{\ast}\mathcal{O}_n$ in terms of the action of $G$ on $h$?
	\end{mybox}
	
	Note that finding the presentation matrix of $f_{\ast}\mathcal{O}_n$ may not be an easy task. For the computations one can use the software {\sc Singular} \rm\cite{singular} and the implementation of Mond-Pellikaan's algorithm given by Hernandes, Miranda, and Pe\~{n}afort-Sanchis in \rm\cite{aldicio}. However, depending on the complexity of $f$, even with the help of a regular computer, the calculations can take days and in some cases cannot be completed due to lack of computer memory. This is the case, for example, of maps whose coordinate functions have high multiplicities. 
	
	In this work we answer both questions. For Question $2$, in a few words, we describe the presentation matrix of $f_{\ast}\mathcal{O}_n$ as a product of three matrices, whose entries depend on the action of $G$ on $h$ (see Theorem \ref{c3}). We note that one importance consequence of obtaining a presentation matrix of $f_*{\cal O}_n$ is the fact that the Fitting ideals give a convenient analytic structure not only for the image of $f$ but also for the image $M_k(f)$ of the (source) multiple points $D^k(f)$ of $f$ (see \cite{[12]}). 
	
	For Question $1$, we consider a reflection group $G$ of order $d$ and a reflected graph map $f(\textit{\textbf{x}})=(w(\textit{\textbf{x}}),h(\textit{\textbf{x}}))$ as in Question $1$. We present the following result (see Theorem \ref{teeeo}).
	
	\begin{theorem}\label{teeeo1} A defining equation of the image of $f(\textit{\textbf{x}})=(w(\textit{\textbf{x}}),h(\textit{\textbf{x}}))$ (given by the $0$-Fitting ideal of the presentation matrix \textcolor{blue}{of} $f_*{\cal O}_n$) is given by the following alternating sum 
		\begin{equation}\label{ima}
			F(X_1,...,X_n,Z)= Z^d-Q_{d-1}Z^{d-1}+Q_{d-2}Z^{d-2}+\cdots+(-1)^{d-1}Q_1Z+(-1)^dQ_0.
		\end{equation}
	\end{theorem}
	
	The description of $Q_{d-k}$ in Theorem \ref{teeeo1} is given in terms of a $G$-invariant polynomial $q_{d-k}(x_1,...,x_n)$ which can be seen as a symmetric polynomial in the variables $h_1,h_2,...,h_d$, where $h_i$ denotes the action of an element $g_i$ of $G$ on $h$ (see Lemma \ref{LEMMA}). We remark that another way to obtain the defining equation of the image of a reflection map is presented in \cite{[rafael]} using sections of the orbit map which is a different technique from the one we use in this work. 
	
	As a corollary of Theorem \ref{teeeo1} we obtain an upper bound for the multiplicity of $f$, where $f$ is a reflected graph map germ. More precisely, write
	
	\begin{equation*}
		F=F_m+F_{m+1}+\cdots+F_k+\cdots
	\end{equation*}
	
	\noindent where each $F_k$ is a homogeneous polynomial of degree $k$ and $F_m \neq 0$. The integer $m$ is called the multiplicity of $\textbf{V}(F)$ in 0 and is denoted by $m(\textbf{V}(F),0)$. Note that by the defining equation (\ref{ima}) in Theorem \ref{teeeo1} we obtain as a corollary that $m(f(\mathbb{C}^n),0)\leq d$, where $d$ is the order of the reflection group $G$. Now, consider a reflection map germ (not necessarily a reflected graph one) $f:(\mathbb{C}^n,0)\rightarrow (\mathbb{C}^{n+1},0)$. A natural question is:
	
	\begin{mybox}
		\textbf{Question 3:} For a fixed (representation of) the group $G$, consider a reflection map germ $f:(\mathbb{C}^n,0)\rightarrow (\mathbb{C}^{n+1},0)$. How big can the multiplicity of the image of $f$ be? In other words, is there an upper bound for the multiplicity of the image of a reflection map?
	\end{mybox}
	
	In this work, we give an answer to Question $3$. An important set of invariants of a reflection group $G$ acting on $\mathbb{C}^{n+1}$ is the set of degrees of $G$, denoted by $\{d_1,d_2,...,d_{n+1}\}$ (see \cite[Prop. 3.25]{[2]}). Ordering the degrees of $G$ in the form $d_1\leq d_2 \leq \cdots \leq d_{n+1}$, we show (see Theorem \ref{teo5.2}) that 
	
	\begin{center}
		$d_1 d_2 \cdot ... \cdot d_n \leq m(f(\mathbb{C}^n),0) \leq d_2 d_3 \cdot ... \cdot d_{n+1}$.
	\end{center}
	
	Finally, as an application of our results, in the case of reflected graph map germs we present another way to obtain a defining equation of the double point hypersurface of $f$ different from the one obtained by in Borges Zampiva, Peñafort-Sanchis, Oréfice Okamoto and Tomazella in \cite[Th. 5.2]{[rafael]}. 
	
	In order to illustrate our results we introduce in Section \ref{dihedral} a new class of map germs called ``\textit{$2m$-dihedral map germs}''. We show the presentation matrix \textcolor{blue}{of} the pushforward $f_{\ast}\mathcal{O}_2$ and a defining equation of the image of a $6$-dihedral map germ in an explicit way (see Proposition \ref{dihedralprop}). In \cite[Th. 3.4]{[3]}, Marar and Nuño-Ballesteros showed that there are no finitely determined quasihomogeneous (with distinct weights) double fold map germs, i.e, $f$ is a reflection graph map germ with reflection group $\mathbb{Z}_2 \times \mathbb{Z}_2$. We finish this work presenting an extension of this result to the group $\mathbb{Z}_r \times \mathbb{Z}_s$, with $r,s\geq 2$ (see Lemma \ref{homo}).
	
	\section{Preliminaries}\label{sec1}
	
	$ \ \ \ \ $  Throughout the text, we assume that $f:(\mathbb{C}^n,0)\rightarrow(\mathbb{C}^{n+1},0)$ is a finite holomorphic map germ, unless otherwise stated. Moreover, $\textit{\textbf{x}}=(x_1,\cdots,x_n)$ and $(\textit{\textbf{X}},Z)=(X_1,\cdots,X_n,Z)$ are used to denote systems of coordinates in $\mathbb{C}^n$ (source) and $\mathbb{C}^{n+1}$ (target), respectively. We also use the standard notation of singularity theory as the reader can find in \cite{[7]}.
	\subsection{Reflection groups and reflection maps}\label{prel}
	
	$ \ \ \ \ $ Consider $GL(\mathbb{C}^p)$ the group of all invertible linear transformations of $\mathbb{C}^p$. Let $Id$ be the identity element of $GL(\mathbb{C}^p)$. A linear representation of a group $G$ with representation space $\mathbb{C}^p$ is a homomorphism $\psi:G \longrightarrow GL(\mathbb{C}^p)$. If $\psi: G \longrightarrow GL(\mathbb{C}^p)$ is a representation, we say that $G$ acts on $\mathbb{C}^p$ and we call $\mathbb{C}^p$ a $G$-module \cite{[serre]}. The action of $g\in G$ on $v \in \mathbb{C}^p$ is defined by $gv := \psi(g)v$ and we usually omit $\psi$ but denote $gv$ by $g_{\bullet}v$. For $g \in GL(\mathbb{C}^p)$, we have $Fix\,g:=\{v\in \mathbb{C}^p\,|\, g_{\bullet}v=v\}$.
	
	\begin{definition}\label{reflection} A reflection on $\mathbb{C}^p$ is a linear map $g:\mathbb{C}^p \longrightarrow \mathbb{C}^p$, satisfying:
		\begin{itemize}
			\item[$(i)$] $g$ is unitary.
			\item[$(ii)$] $g$ has finite order.
			\item[$(iii)$] $\dim Fix \ g=p-1$.
		\end{itemize}
	\end{definition}
	
	Let $U(\mathbb{C}^p)$ be the group of unitary automorphisms of $\mathbb{C}^p$. A subgroup $G$ of $U(\mathbb{C}^p)$ is said to be a reflection group if it is generated by reflections. If $g$ is a reflection, the subspace $Fix\, g$ is a hyperplane, called the \textit{reflecting hyperplane} of $g$.
	
	We note that the action of $G$ on $\mathbb{C}^p$ induces an action of $G$ on a holomorphic function, as we will see in the next definition.
	
	\begin{definition} Let $g \in GL(\mathbb{C}^p)$ and $P \in {\cal O}_p$ be a holomorphic function. We define the action of $g$ on $P$ by
		\begin{equation*}
			(g_{\bullet}P)(v):=P(g^{-1}(v)), \textrm{ for all } v\in \mathbb{C}^p.
		\end{equation*}
	\end{definition}
	
	We note that for any $g \in G$ the action of $g$ on $P,Q \in \mathcal{O}_p$ has the following property: $g_{\bullet}(PQ)=g_{\bullet}(P)g_{\bullet}(Q)$. We say that $P\in S$ is \textit{$G$-invariant} if $g_{\bullet}P=P$ for all $g \in G$. The \textit{algebra of invariants of $G$} is the algebra of $G$-invariant holomorphic functions
	
	\begin{center}
		$ {\cal O}_p^G=\{P \in  {\cal O}_p\, |\, g_{\bullet}P=P \textrm{ for all } g \in G\}.$
	\end{center}
		
	The following lemma can be found for instance in \cite[Lemma 3.17]{[2]}. It gives us a way to obtain the action of a reflection of $G$ on an element of ${\cal O}_p$.
	
	\begin{lemma}\label{li} If $g$ is a reflection in $GL(\mathbb{C}^p)$ and if $Fix\,g$ is its reflecting hyperplane, with $Fix\,g=Ker\,L_g$. Then for all $P \in {\cal O}_p$ there exists $Q \in {\cal O}_p$ such that
		\begin{equation*}
			g_{\bullet}P=P+L_{g}Q.
		\end{equation*}
	\end{lemma}
	
	It follows from Lemma \ref{li} that if $g$ is a reflection, then $\dfrac{P-g_{\bullet}P}{L_g}=-Q \in {\cal O}_p$. So we can consider the operator $\Delta_{g}:{\cal O}_p \longrightarrow {\cal O}_p$ given by
	\begin{equation*}
		\Delta_{g}(P)=\dfrac{P-{g}_{\bullet}P}{L_g}
	\end{equation*}
	where $g$ is a reflection of $G$. This operator is also known as \textit{Demazure operator}.\\
	
	Consider now the operator $Av:{{\cal O}_p} \longrightarrow {\cal O}_p^G$, given by
	\begin{equation*}
		Av(P):=\dfrac{1}{|G|}\left( \sum_{g \in G} g_{\bullet}P\right).
	\end{equation*}

	This operator is also known as \textit{Reynolds operator}. It is clear from the definition that $Av(P) \in {\cal O}_p^G$ and that $Av(P)$ is either $0$ or has the same degree as $P$. Moreover, for $P \in {\cal O}_p^G$ we have $Av(P) = P$ and therefore $Av^2 = Av$. Thus $Av$ is a projection of ${\cal O}_p$ onto ${\cal O}_p^G$. In fact, a somewhat stronger statement is true, namely that for $P \in {\cal O}_p^G$ and $Q \in S$ we have $Av(PQ) = P Av(Q)$ so that $Av$ is a ${\cal O}_p^G$-module homomorphism.\\
	
	The orbit map $w$ of a group $G$ acting on $\mathbb{C}^p$ determines a way of ``folding''\,$\mathbb{C}^p$, gluing an orbit of $G$ to a point \cite{[5]}. Figure \ref{figura6} illustrates a geometrical idea of the notion of a reflection map.
	
	\begin{figure}[H]
		\centering
		\includegraphics[width=0.7\linewidth]{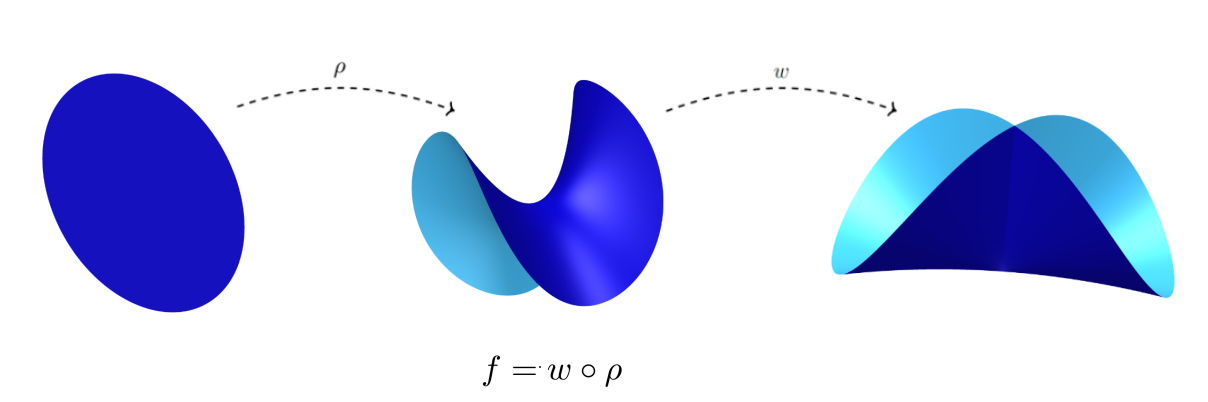}
		\caption{Illustration of a reflection map}\label{figura6}
		%	\label{fig:novo}
	\end{figure}
	
	Since $G$ acts on $\mathbb{C}^p$ then by Shephard-Todd's Theorem (\cite[Th. 5.1]{[6]}, see also \cite[Th. 3.20]{[2]}) we have that the algebra of $G$-invariant holomorphic functions ${\cal O}_p^G$ can be generated by $p$ homogeneous holomorphic functions.
	
	\begin{definition}  The orbit map of a reflection group $G$ is a map $w:\mathbb{C}^p \longrightarrow \mathbb{C}^p$ whose coordinate functions are homogeneous polynomials $w_1,w_2,... ,w_p$ in $S$ that generate $J$. The set of degrees $d_1,\cdots,d_p$ of $G$ are the degrees of $w_1,w_2,\cdots,w_p$, respectively.
	\end{definition} 
	
	The map $w$ is unique up to invertible polynomial transformations in the target. Furthermore, it is well known that the degrees of $G$ do not depend on the choice of the set of generators $w_1,\cdots,w_p$ of ${\cal O}_p^G$. Since we work with objects which are invariant under such transformations, the choice of $w$ does not matter. This justifies our calling $w$ the ``orbit map'' of $G$.
	
	\begin{remark}\label{orbitinvariant} From \cite[Th. 3.20]{[2]}, we know that ${\cal O}_p^G \simeq \mathbb{C}\{w_1,\ldots,w_p\}$, which means that for all $g \in G$ and all $P \in \mathbb{C}\{w_1,\ldots,w_p\}$, we have $g_{\bullet}P=P$.
	\end{remark}
	
	\begin{example}\label{example1} Consider the generators of the dihedral group ${D}_{2m}$ in $GL(\mathbb{C}^2)$,
		$$R=\left[ \begin{array}{rrrrr}
			0 & 1 \\
			1 & 0
		\end{array}\right] \hspace{0.5cm} \textrm{and} \hspace{0.5cm} S=\left[ \begin{array}{rrrrr}
			e^{\frac{2\pi i}{m}} & 0 \\
			0 & e^{-\frac{2\pi i}{m}} 
		\end{array}\right].$$
		
		The orbit map $w:\mathbb{C}^2 \longrightarrow \mathbb{C}^2$ of $D_{2m}$ is given by
		\begin{equation*}
			(x,y) \longmapsto (xy, \ x^m+y^m).
		\end{equation*}
	\end{example}
	
	One of the most important results about the orbit map is Noether's Theorem \cite{[4]} which allows us to conclude that the set $w^{-1}(w(v))$ is the orbit of $G$ in $v$, i.e, the set $Gv:=\lbrace  g_{\bullet}v \ | \ g \in G \rbrace $.
	
	\begin{theorem}\label{teo} \textbf{[Noether]} For any $v \in \mathbb{C}^p$, $w^{-1}(w(v))=Gv$.
	\end{theorem} 
	
	Considering $G=\mathbb{Z}_4$, Figure \ref{figura7} illustrates Noether's Theorem. 
	
	\begin{figure}[H]
		\begin{center}
			\begin{tikzpicture}[scale=0.9]
				\draw[fill=black!60,nearly transparent, thick] (0,0) to (4,0) to (3,-2) to (-1,-2) to (0,0);
				\draw[black,dashed] (2,0) to (1, -2);
				\draw[black,dashed] (-0.5,-1) to (3.5, -1);
				
				\node at (0.7,-0.4) {\color{black}{\small $\bullet$}};
				\node at (2.7,-0.4) {\color{black}{\small $\bullet$}};
				\node at (2.2,-1.6) {\color{black}{\small $\bullet$}};
				\node at (0.2,-1.6) {\color{black}{\small $\bullet$}};
				
				\draw[fill=black!60,nearly transparent, thick] (6,0) to (8,0) to (7,-1.5) to (5,-1.5) to (6,0);
				
				\node at (6.8,-1.3) {\color{black}{\small $\bullet$}};
				\draw[black,dashed] (5,-1.7) to (7, -1.7);
				\draw[black,dashed] (7.1,-1.7) to (8.2, 0);
				\draw[black,dashed] (5,-1.9) to (7.1, -1.9);
				\draw[black,dashed] (7.2,-1.9) to (8.4, 0);
				\draw[black,dashed, thick][->] (2.3, 0.3) to [out=20,in=-200] (6.5, 0.3);
				\node at (4.4,1) {\color{black}{\small $w$}};
				
				\node at (0,-1.3) {\color{black}{\small $Idv=v$}};
				\node at (6.8,-1) {\color{black}{\small $w(v)$}};
				
				\draw[black,dashed, thick][<-] (2.3, -2.3) to [out=-20,in=200] (6.5, -2.3);
				\node at (4.4,-3) {\color{black}{\small  $w^{-1}(w(v))$}};
				\node at (2,-1.4) {\color{black}{\small $g_2v$}};
				\node at (2.5,-0.2) {\color{black}{\small $g_3v$}};
				\node at (0.5,-0.2) {\color{black}{\small $g_4v$}};
			\end{tikzpicture}
		\end{center}
		\caption{Orbit map of a reflection group.}\label{figura7}
	\end{figure}
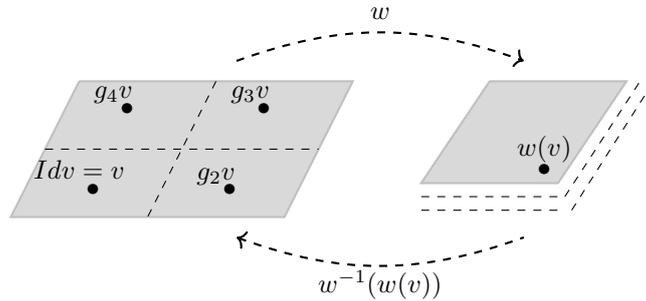
	
	\begin{remark}\label{obs2} It is important to mention that whether a group is a reflection group or not depends on its representation in $GL(\mathbb{C}^p)$. For example, the cyclic group $\mathbb{Z}_d$ generated by the matrix
		
		$$K=\left[ \begin{array}{rrrrr}
			1 & 0 \\
			0 & e^{\frac{2\pi i}{d}}
		\end{array}\right].$$
		\noindent is a reflection group, since $K$ is a reflection. On the other hand, the group generated by the matrix
		
		$$K^{'}=\left[ \begin{array}{rrrrr}
			e^{\frac{2\pi i}{d}} & 0 \\
			0 & \left( e^{\frac{2\pi i}{d}}\right) ^k
		\end{array}\right],$$
		\noindent with $gcd(k,d)=1$, is isomorphic to $\mathbb{Z}_d$. 
		
		Suppose that $\langle K^{'} \rangle \simeq \mathbb{Z}_d$ is a reflection group. By \rm\cite[\textit{Section} \rm $8$]{[6]} \textit{(see also} \rm\cite[\textit{Th.} \rm $4.14$]{[2]}\textit{) we have that the number of reflections in $\mathbb{Z}_d$ is $d-1$ (obviously the identity matrix is not a reflection). Therefore, every element $g \in \langle K^{'} \rangle$ with $g \neq Id$ is a reflection. Note that $Fix K^{'}=0$ (which has dimension 0). By Definition} \rm\ref{reflection}, \textit{$K^{'}$ is not a reflection, a contradiction. Therefore, $\langle K^{'} \rangle$ cannot be generated by reflections.}
	\end{remark}
	
	%\vspace{0.2cm}
	
Let $G$ be a reflection group acting on $\mathbb{C}^p$. Let $Fix \ g_1, \cdots, Fix \ g_l$ be all the (distinct) reflecting hyperplanes of $G$ (where $g_i$ is a reflection of $G$). Let $ord(g_i)$ be the order of $g_i$ and let $L_{g_1},...,L_{g_{\ell}}$ be linear forms such that $Fix \ g_i:=Ker \,L_{g_i}$. An important relation between the orbit map and the defining equations of the reflecting hyperplanes of $G$ is given by the following proposition which can be found for instance in \cite[Th. 9.8]{[6]}. We shall denote the Jacobian matrix of the orbit map $w$ by $jac(w)$.
	
	\begin{proposition}\label{li3} For some non-zero constant $c$ we have that
		\begin{equation*}
			\det(jac(w))=c\prod_{i=1}^{\ell} L_{g_i}^{ord(g_i)-1}.
		\end{equation*}
		
	\end{proposition}
	
	\begin{definition}\label{0} Let $G$ be a reflection group acting on $\mathbb{C}^p$ and $h:\mathbb{C}^p\longrightarrow \mathbb{C}^r$ be any holomorphic map.
		
		\noindent (a) A $G$-reflection map $f:\mathbb{C}^n\longrightarrow \mathbb{C}^{p}$ is a map given by the composition of an embedding $\rho: \mathbb{C}^n\hookrightarrow \mathbb{C}^{p}$ with the orbit map $w: \mathbb{C}^{p}\longrightarrow \mathbb{C}^{p}$ of $G$, i.e., $f=w \circ \rho$.

		\noindent (b) The $G$-reflected graph of $h$ is the map $(w,h):\mathbb{C}^p \longrightarrow \mathbb{C}^{p+r}$, given by $f(\textbf{x})=(w(\textbf{x}),h(\textbf{x}))$.
		
	\end{definition}
	
	The $G$-reflected graph $f=(w,h)$ is the $G$-reflection map obtained by taking the graph embedding $h$, given by $\textit{\textbf{x}} \longmapsto(\textit{\textbf{x}}, h(\textit{\textbf{x}}))$, and letting $G$ act on $\mathbb {C}^p \times \mathbb{C}^r$, trivially on the second factor (a trivial extension of the action to $\mathbb{C}^{p+r}$). To simplify, we will omit $G$ and simply say reflection map, when $G$ is clear in the context. Throughout this work we will only consider $G$-reflected graphs in the case where $r=1$. For a $G$-reflection map, we will consider only the case where $p=n+1$.
	
	\section{The presentation matrix of reflection maps}\label{secpres}
	
	$ \ \ \ \ \ $ In this section, we will consider a reflection group $G$ acting on $\mathbb{C}^n$ and a holomorphic function $h:\mathbb{C}^n\longrightarrow \mathbb{C}$. We consider an extension of the action of $G$ on $\mathbb{C}^n \times \mathbb{C}$, simply making $G$ act trivially on the second factor, in this sense we also consider an orbit map $w=(w_1,\cdots,w_n,Z)$. We will study a reflection graph from $(\mathbb{C}^n,0)$ to $(\mathbb{C}^{n+1},0)$ given by $f=(w_1,\cdots,w_n,h)$. In this context, we will provide an answer to Question $2$ in the Introduction, i.e., we will present a presentation matrix of the pushforward $f_{\ast}\mathcal{O}_n$ as an $\mathcal{O}_{n+1}$-module, via composition with $f$, where $f$ is a reflected graph map germ. We note that our way to construct the presentation matrix depends only on the action of $G$ on $h$.\\ 
	
	Before we do this, we introduce some notation. Let $f:(\mathbb{C}^n,0)\rightarrow (\mathbb{C}^{n+1},0)$ be a finite analytic map germ, $r_1,r_2,\cdots,r_{\ell}$ be elements in $\mathcal{O}_n$ and denote the maximal ideal of $\mathcal{O}_{n+1}$ by $\textbf{m}$. Note that $\mathcal{O}_n$ is a finite $\mathcal{O}_{n+1}$-module via $f$ and if the classes of $r_1,\cdots,r_{\ell}$ in $\mathcal{O}_n/f^{\ast}\textbf{m}$ generate it as a vector space over $\mathbb{C}=\mathcal{O}_{n+1}/\textbf{m}$, then $r_1,\cdots,r_{\ell}$ generate $\mathcal{O}_n$ as $\mathcal{O}_{n+1}$-module via $f$ (see e.g. \cite[Cor. D.1 and D.2]{[7]}). It is clear that we can take $r_1=1$, and we shall do so from now on. A presentation of $\mathcal{O}_n$ over $\mathcal{O}_{n+1}$ is an exact sequence
	
	\begin{equation}\label{eq8}
		\mathcal{O}_{n+1}^{r} \stackrel{\lambda}{\longrightarrow} \mathcal{O}_{n+1}^{s} \stackrel{\psi}{\longrightarrow} \mathcal{O}_n \longrightarrow 0
	\end{equation}
	
	\noindent of $\mathcal{O}_{n+1}-$modules. It follows by \cite[Lemma 2.1]{[12]} that $r=s=\ell$ in (\ref{eq8}). We say that $\lambda$ in (\ref{eq8}) is the \textit{presentation matrix} of $f_{\ast}\mathcal{O}_n$. Mond and Pellikaan also give an algorithm to construct a presentation (for details, see \cite[Section 2.2]{[12]}).\\
	
	Let $G= \lbrace g_1,\cdots,g_d \rbrace $ be a reflection group of order $d$ acting on $\mathbb{C}^n$. We will adopt the notation $\textit{\textbf{w}}=(w_1,w_2,...,w_n)$ for the orbit map of $G$, $(\textit{\textbf{X}},Z)=(X_1,\cdots,X_{n},Z)$ for the coordinates of $\mathbb{C}^{n+1}$ (target) and $\textit{\textbf{x}}=(x_1,x_2,...,x_n)$ for the coordinates of $\mathbb{C}^n$ (source). Throughout this work, the bold notation (e.g. $\textit{\textbf{w}}, \textit{\textbf{x}}$ and $\textit{\textbf{X}}$) will always be used to represent an $n$-tuple. Recall that, according to Definition \ref{reflection}, an element $g \in G$ is a \textit{reflection} on $\mathbb{C}^n$ if $g$ is unitary, has finite order, and satisfies $\dim Fix\, g=n-1$. We denote by ${\cal R}$ the set of all reflections of $G$.\\ 
	
	Recall that a reflected graph map germ $f:(\mathbb{C}^n,0)\rightarrow (\mathbb{C}^{n+1},0)$, $f(\textit{\textbf{x}})=(w(\textit{\textbf{x}}),h(\textit{\textbf{x}}))$, is obtained as the composition $f=w \circ \rho$, where $\rho$ is the embedding $(\textit{\textbf{x}},h(\textit{\textbf{x}}))$ and $w$ is the orbit map of $G$. Denote by ${\cal Y}\coloneqq\rho(\mathbb{C}^n)$ and ${\cal X}\coloneqq w({\cal Y})$. Note that ${\cal X}=f(\mathbb{C}^n)=\textbf{V}(F(\textit{\textbf{X}},Z))$, where $F(\textit{\textbf{X}},Z)$ is a defining equation of the image of $f$ which is obtained as the determinant of the presentation matrix of $f_{\ast}(\mathcal{O}_{n})$. It follows that a defining equation of $(\mathcal{Y},0)$ is $Z-h(X_1,\cdots,X_n)=0$. By Noether's Theorem (see Theorem \ref{teo}) we have that 
	\begin{equation*}
		w^{-1}(w({\cal Y}))=\displaystyle \bigcup_{i=1}^d {g_{i}}_{\bullet}\mathcal{Y}.
	\end{equation*} The pre-image of $\mathcal{X}$ by $w$ consists of the orbit of $\mathcal{Y}$ under the action of $G$. We will show in the following lemma that a defining equation of $w^{-1}(w({\cal Y}))$ is  given by 
	\begin{equation}\label{eqY}
		(Z-{g_1}_{\bullet}h(\textit{\textbf{X}}))\cdot (Z-{g_2}_{\bullet}h(\textit{\textbf{X}})) \cdots (Z-{g_d}_{\bullet}h(\textit{\textbf{X}}))=0.
	\end{equation}
	The following Lemma will be a key tool to prove Theorem \ref{teeeo} where we will provide a defining equation of the image of $f$. We will also show that the defining equation of $w^{-1}(w({\cal Y}))$, described in (\ref{eqY}), coincides with the pullback of the defining equation $F(\textit{\textbf{X}},Z)$ of $(\mathcal{X},0)$ by $\textit{\textbf{w}}$, which will be denoted by $F(w_1,\cdots,w_n ,Z)=F(\textit{\textbf{w}},Z)$.

	\begin{lemma}\label{prim} With the notation above, let $f:(\mathbb{C}^n,0) \longrightarrow (\mathbb{C}^{n+1},0)$, $f(\textbf{x})=(w(\textbf{x}),h(\textbf{x}))$ be a reflected graph map germ, then $w^{-1}(w({\cal Y}))=\textbf{V}(F(\textbf{w},Z))$. Furthermore,
		\begin{equation}\label{factors}
			F(\textbf{w},Z)=\prod_{k=1}^{d}\left(Z-{g_k}_{\bullet}(h(\textbf{x}))\right).
		\end{equation}
	\end{lemma}
	
	\begin{proof}Note that $({\cal X},0)\subset (\mathbb{C}^{n+1},0)$ is defined by a single equation (the determinant of the presentation matrix $\lambda[\textit{\textbf{X}},Z]$ of $f_*{\cal O}_n$ as a ${\cal O}_{n+1} $-module via $f$, denoted by $F(\textit{\textbf{X}},Z)$). Since $F(\textit{\textbf{X}},Z)$ is a defining equation of $\mathcal{X}$, it follows that the pullback $F(\textit{\textbf{w}},Z)$ of $F(\textit{\textbf{X}},Z)$ by $w$ is a defining equation of $w^{-1}(w(\mathcal{Y}))$. Let us show that $F(\textit{\textbf{w}},Z)$ can be factored as in (\ref{factors}).\\
		
		By \cite[Cor. 3.29]{[2]} we have that 
		
		\begin{equation*}
			\dim_{\mathbb{C}} \dfrac{\mathbb{C}\{x_1,x_2,...,x_n\}}{\left\langle w_1,w_2,...,w_n\right\rangle }=|G|=d.
		\end{equation*}
		
		\noindent Let $1=r_1,r_2,...,r_{d}$ be a basis of $\dfrac{\mathbb{C}\{\textit{\textbf{x}}\}}{\left\langle \textit{\textbf{w}}\right\rangle }$ as a $\mathbb{C}$-vector space. By \cite[Lemma 3.28]{[2]} there are $p_1,p_2,\cdots, p_d$ in $\mathbb{C}\{\textit{\textbf{x}}\}$ such that 
		
		\begin{equation*}
			h:=r_1p_1(w)+r_2p_2(w)+...+r_{d}p_{d}(w),
		\end{equation*}
		Note that $p_i(\textit{\textbf{w}})=p_i(w_1,\cdots,w_n)$ are uniquely determined by $h$ and the orbit map $w$. Now we will follow Mond-Pellikaan's algorithm (see \cite[Section 2.2]{[12]}) to construct a presentation matrix $\lambda[\textit{\textbf{X}},Z]$ of $f_*{\cal O}_n$ as an ${\cal O}_{n+1} $-module via $f$. Let us first consider the matrix
		
		$$\alpha[\textit{\textbf{X}}]=\left[ \begin{array}{rrrrr}
			\alpha_{1,1} & \alpha_{1,1} & \alpha_{1,2} & \cdots & \alpha_{1,d} \\
			\alpha_{2,1} & \alpha_{2,2} & \alpha_{2,3} & \cdots & \alpha_{2,d} \\
			\alpha_{3,1} & \alpha_{3,2} & \alpha_{3,3} & \cdots & \alpha_{3,d} \\
			\vdots & \vdots & \vdots & \cdots & \vdots \\
			\alpha_{d,1} & \alpha_{d,2} & \alpha_{d,3} & \cdots & \alpha_{d,d}
		\end{array}\right]_{d\times d},$$
		
		\noindent where the $\alpha_{i,j}(\textit{\textbf{X}})$ are in $\mathbb{C}\{\textit{\textbf{X}}\}$ and satisfy the relation
		
		\begin{equation}\label{relation}
			r_i \cdot h=\sum_{j=1}^{d}(\alpha_{i,j}(\textit{\textbf{w}}))\cdot r_j, \ \ \ \ \  with \ \ \ \ i=1,\cdots, d.
		\end{equation}
		
		Again by \cite[Lemma 3.28]{[2]} we note that the $\alpha_{i,j}$ are uniquely determined by $h$ and $w$. Thus, the presentation matrix of $f_*{\cal O}_n$ is
		 
		$$\lambda[\textit{\textbf{X}},Z]=\alpha[\textit{\textbf{X}}]-ZId=\left[ \begin{array}{rrrrr}
			\alpha_{1,1}-Z & \alpha_{1,2} & \alpha_{1,3} & \cdots & \alpha_{1,d} \\
			\alpha_{2,1} & \alpha_{2,2}-Z & \alpha_{2,3} & \cdots & \alpha_{2,d} \\
			\alpha_{3,1} & \alpha_{3,2} & \alpha_{3,3}-Z & \cdots & \alpha_{3,d} \\
			\vdots & \vdots & \vdots & \cdots & \vdots \\
			\alpha_{d,1} & \alpha_{d,2} & \alpha_{d,3} & \cdots & \alpha_{d,d}-Z
		\end{array}\right].$$
		
		Having defined the presentation matrix $\lambda[\textit{\textbf{X}},Z]$ according to Mond-Pellikaan's algorithm, now we would like to factorize $F(\textit{\textbf{w}},Z)$ as in (\ref{factors}). Substituting $\textit{\textbf{X}}$ for $\textit{\textbf{w}}$ in $\lambda[\textit{\textbf{X}},Z]$ and $\alpha[\textit{\textbf{X}}]$, we obtain the matrices ${\lambda} [\textit{\textbf{w}},Z]$ and $\alpha[\textit{\textbf{w}}]$, i.e., ${\lambda} [\textit{\textbf{w}},Z]$ ($\alpha[\textit{\textbf{w}}]$, respectively) is the pullback of $\lambda[\textit{\textbf{X}},Z]$ ($\alpha[\textit{\textbf{X}}]$, respectively) by $w$. Clearly, the determinant of ${\lambda} [\textit{\textbf{w}},Z]$ is equal to $F(\textit{\textbf{w}},Z)$.
		
		Let $\mathbb{K}:=Frac(\mathbb{C}\{\textit{\textbf{x}}\})$ and consider $\alpha[\textit{\textbf{w}}]$ as a matrix with entries in $\mathbb{K}$. Note that ${\lambda} [\textit{\textbf{w}},Z]=(\alpha[\textit{\textbf{w}}]-ZId)$, where $Id$ is the identity matrix. Therefore $F(\textit{\textbf{w}},Z)$ is precisely the characteristic polynomial of $\alpha[\textit{\textbf{w}}]$. Let us find the eigenvalues of $\alpha[\textit{\textbf{w}}]$. Consider $Id=g_1,g_2,...,g_{d}$ the elements of the reflection group $G$. For all $i=1,\cdots,d$ we obtain from Equation (\ref{relation}) that
		\begin{equation}\label{auto}
			\alpha_{i,1}(\textit{\textbf{w}})\cdot r_1+\cdots+\alpha_{i,d}(\textit{\textbf{w}}) \cdot r_{d}=h\cdot r_i.
		\end{equation}
		
		Since each $\alpha_{i,j}(\textit{\textbf{w}})=\alpha_{i,j}(w_1,\cdots,w_n)$ is in  $\mathbb{C}\{w_1,\ldots,w_p\}$, by Remark \ref{orbitinvariant} we have that $\alpha_{i,j}(\textit{\textbf{w}})$ is invariant under the action of $G$ for all $i,j \in \{1,...,n\}$. Thus, applying the action of $g_k \in G$ in (\ref{auto}), we obtain that

		\begin{equation*}
			\alpha_{i,1}(\textit{\textbf{w}})\cdot {g_k}_{\bullet}(r_1)+\cdots+\alpha_{i,d}(\textit{\textbf{w}}) \cdot {g_k}_{\bullet}(r_{d})= {g_k}_{\bullet}(h) \cdot {g_k}_{\bullet}(r_i)
		\end{equation*} 
		from which we conclude that 
		$$\left[ \begin{array}{ccccc}
			\alpha_{1,1} & \alpha_{1,2} & \alpha_{1,3} & \cdots & \alpha_{1,d} \\
			\alpha_{2,1} & \alpha_{2,2} & \alpha_{2,3} & \cdots & \alpha_{2,d} \\
			\alpha_{3,1} & \alpha_{3,2} & \alpha_{3,3} & \cdots & \alpha_{3,d} \\
			\vdots & \vdots & \vdots & \cdots & \vdots \\
			\alpha_{d,1} & \alpha_{d,2} & \alpha_{d,3} & \cdots & \alpha_{d,d}
		\end{array}\right]\cdot \left[ \begin{array}{c}
			{g_k}_{\bullet}(r_1) \\
			{g_k}_{\bullet}(r_2) \\
			{g_k}_{\bullet}(r_3) \\
			\vdots  \\
			{g_k}_{\bullet}(r_{d})
		\end{array}\right]={g_k}_{\bullet}(h) \left[ \begin{array}{c}
			{g_k}_{\bullet}(r_1)  \\
			{g_k}_{\bullet}(r_2)\\
			{g_k}_{\bullet}(r_3) \\
			\vdots \\
			{g_k}_{\bullet}(r_{d})
		\end{array}\right].$$
		
		In this way for all $k=1,\cdots,d$ we have that $({g_k}_{\bullet}(r_1),{g_k}_{\bullet}(r_2),...,{g_{k}}_{\bullet}(r_{d}))$ is an eigenvector of $\alpha[\textit{\textbf{w}}]$ with respective eigenvalue ${g_k}_{\bullet}(h)$. 
		
		Consider the matrix $E$ whose columns are the eigenvectors $({g_k}_{\bullet}(r_1),{g_k}_{\bullet}(r_2),...,{g_{k}}_{\bullet}(r_{d}))$, i.e.
		
		$$E=\begin{bmatrix}
			{g_1}_{\bullet}r_1 & {g_2}_{\bullet}r_1 &  \cdots & {g_d}_{\bullet}r_1 \\
			{g_1}_{\bullet}r_2 & {g_2}_{\bullet}r_2 &  \cdots & {g_d}_{\bullet}r_2 \\
			\vdots & \vdots &  \ddots & \vdots \\
			{g_1}_{\bullet}r_{d-1} & {g_2}_{\bullet}r_{d-1} &  \cdots & {g_d}_{\bullet}r_{d-1} \\
			{g_1}_{\bullet}r_d & {g_2}_{\bullet}r_d &  \cdots & {g_d}_{\bullet}r_d \
		\end{bmatrix}$$
		
		It follows by Gutkin's Theorem (see \cite[Th. 10.13]{[2]}, where the matrix $A_{i,j}$ in Definition 10.6 of \cite{[2]} is the matrix $E$ in our setting) that the determinant of $E$ is a non-zero. Hence, the set of these eigenvectors are linearly independent. Therefore,
		
		\begin{equation}\label{det}
			\det \lambda[\textit{\textbf{w}},Z]:=F(\textit{\textbf{w}},Z)=\prod_{k=1}^{d}\left(Z-{g_k}_{\bullet}(h(\textit{\textbf{x}}))\right).
		\end{equation}
		Which completes the proof.\end{proof}

	\begin{remark}\label{obs}(a) In the sequel, given a matrix $M=(m_{i,j}(w_1,\cdots,w_n))$ where each $m_{i,j}$ is in $\mathbb{C}[w_1,\cdots,w_n]$ (i.e., $m_{i,j}$ is $G$-invariant). Also, we denote by $w^{\ast}(M)$ the matrix obtained by exchanging $w_i$ for $X_i$ in $m_{i,j}$, i.e., $w^{\ast}(M)=(m_{i,j}(X_1,\cdots,X_n))$.\\

		\noindent (b)	Let $\lambda[\textbf{X},Z]$ be the presentation matrix of $f_{*}{\cal O}_n$ as a ${\cal O}_{n+1}$-module and $\lambda[\textbf{w},Z]$ as in Lemma \ref{prim}. With the notation above consider the matrices $E$ and $A$ below
		
		$$E=\begin{bmatrix}
			{g_1}_{\bullet}r_1 & {g_2}_{\bullet}r_1 &  \cdots & {g_d}_{\bullet}r_1 \\
			{g_1}_{\bullet}r_2 & {g_2}_{\bullet}r_2 &  \cdots & {g_d}_{\bullet}r_2 \\
			\vdots & \vdots &  \ddots & \vdots \\
			{g_1}_{\bullet}r_{d-1} & {g_2}_{\bullet}r_{d-1} &  \cdots & {g_d}_{\bullet}r_{d-1} \\
			{g_1}_{\bullet}r_d & {g_2}_{\bullet}r_d &  \cdots & {g_d}_{\bullet}r_d \
		\end{bmatrix} \hspace{0.3cm} \textrm{and} \hspace{0.3cm} A=\begin{bmatrix}
			{g_1}_{\bullet}h & 0 & 0 & \cdots & 0 \\
			0 & {g_2}_{\bullet}h & 0 & \cdots & 0 \\
			\vdots & \vdots & \vdots & \ddots & \vdots \\
			0 & 0 & 0 & \cdots & 0\\
			0 & 0 & 0 & \cdots & {g_d}_{\bullet}h \
		\end{bmatrix}.$$
		
		As we say in the proof of Lemma \ref{prim}, note that $E$ is the eigenvector matrix of $\lambda[\textit{\textbf{w}},0]$ and the elements that appear on the diagonal of the matrix $A$ are the eigenvalues of $\lambda[\textit{\textbf{w}},0]$. This motivates the following theorem.
	\end{remark}
	
	\begin{theorem}\label{c3} Let $f:(\mathbb{C}^n,0)\longrightarrow (\mathbb{C}^{n+1},0)$, $f(\textbf{x})=(w(\textbf{x}),h(\textbf{x}))$, be a reflected graph map germ. Consider the matrices $E$ and $A$ as above and set $ A_Z:=(A-ZId)$. The presentation matrix of $f_{*}{\cal O}_n$ as a ${\cal O}_{n+1}$-module $\lambda[\textbf{X} ,Z]$ is expressed as the product of the matrices $E$, $A$ and $E^{-1}$, i.e.
		
		\begin{equation*}
			\lambda[\textbf{X},Z]=w^{\ast}(E\cdot A_Z \cdot E^{-1}).
		\end{equation*}
		
	\end{theorem}
	
	\begin{proof}From the proof of Lemma \ref{prim} we have that $E^{-1}\lambda[\textit{\textbf{w}},0]E=A$. In other words, the matrix $E$ diagonalizes the matrix $\lambda[\textit{\textbf{w}},0]$. Thus, 
		
		\begin{equation*}
			\lambda[\textit{\textbf{w}},Z]=EAE^{-1}-ZId=E(A-ZId)E^{-1}=E\cdot A_Z \cdot E^{-1}.
		\end{equation*}
		
		\noindent Now, the result follows from the fact that $w^{\ast}( \lambda[\textit{\textbf{w}},Z])= \lambda[\textit{\textbf{X}},Z]$. \end{proof}
	
\begin{definition}The quotient $\mathcal{O}_n / (w_1,\cdots,w_n)$ is called the \textit{coinvariant algebra} of the reflection group $G$ (see \cite[Ch. 3, Sec. 6]{[2]}).
\end{definition}

Another important space in the theory of reflection groups is the space $\mathcal{H}$ of $G$-harmonic analytic functions (see \cite[Def. 9.35]{[2]}). For the purposes of this work, we highlight the fact that $\mathcal{H}$ is isomorphic (as a $G$-module) to the coinvariant algebra of $G$. Thus, in what follows, we can identify $\mathcal{H}$ with $\mathcal{O}_n / (w_1,\cdots,w_n)$ as $G-$modules.	
	
	Note that in the proof of Lemma \ref{prim} we obtained that the determinant of the matrix $E$ is not zero. Actually, we can say more about the determinant of the matrix $E$. First, let us establish some notation. Consider a reflection group $G$ of order $d$ acting on $\mathbb{C}^n \times \mathbb{C}$ (trivially on the second factor). Denote by $Fix\,g_1,\cdots, Fix\,g_l$ the (distinct) reflecting hyperplanes of $G$. Let $L_{g_i}$ be a linear form from $\mathbb{C}^{n+1}$ to $\mathbb{C}$ such that $Fix\,g_i=\textbf{V}(L_{g_i})$ for all $i$. Let $g_{i}$ be a generator of the cyclic group which fixes $Fix\,g_i$ and denote the order of $g_{i}$ by $ord(g_i)$. The following proposition gives us an expression for the determinant of $E$.  
	
	\begin{proposition}\label{det} Let $G$ be a finite reflection group and $f:(\mathbb{C}^n,0)\longrightarrow (\mathbb{C}^{n+1},0)$ be a reflected graph map germ, then for some non-zero constant $c$, we have
		\begin{equation*}
			\det(E)=c\prod_{i=1}^{\ell}{L_{g_i}}^{|G|(ord(g_i)-1)/2}
		\end{equation*}
	\end{proposition}

	\begin{proof} Consider the coinvariant algebra $\mathcal{O}_n/(w_1,\cdots,w_n)$ of $G$. By \cite[Cor. 3.29]{[2]} we have that 
	
	\begin{center}
	$\mathcal{O}_n/(w_1,\cdots,w_n)\simeq \mathbb{C}^{d}$,
	\end{center}
	
\noindent	as $\mathbb{C}$-vector spaces, where $d:=|G|$. Take a basis of the coinvariant algebra $\lbrace r_1,r_2,\cdots,r_d \rbrace $ with $r_1=1$ as usually. By \cite[Prop. 3.2]{[2]} we can see $M:=\mathbb{C}^{d}$ as a $G$-module, where the representation of $G$ on $\mathbb{C}^{d} \simeq \mathcal{O}_n/(w_1,\cdots,w_n)$ is the regular representation (see \cite[Sec. 1.2]{[serre]} for the definition of a regular representation of a finite group). Consider the canonical basis

\begin{center}
$\lbrace e_{g_1},e_{g_2},\cdots, e_{g_d}\rbrace $
\end{center}

\noindent of $\mathbb{C}^d$ indexed by the elements of $G$. The dual $M^{\ast}=Hom_{\mathbb{C}}(M,\mathbb{C})$ has dual basis $\lbrace y_1,\cdots,y_d \rbrace$ where $y_j:=e^{g_j}$ is the linear functional from $\mathbb{C}^d$ to $\mathbb{C}$ defined by $y_j(e_{gk})=\delta_{jk}$ (Kronecker symbol). By \cite[Cor. 9.37]{[2]} we have that $\mathcal{H}$ is isomorphic to the coinvariant algebra of $G$ (as $G$-modules). Furthermore, there is a canonical isomorphism

\[(\mathcal{H}\otimes M^*)^G \;\cong\; \operatorname{Hom}_{\mathbb{C}G}(M,\mathcal{H}),\]

\noindent sending $\sum_\ell A_\ell\otimes \phi_\ell$ to the linear operator $v\mapsto \sum_\ell \phi_\ell(v)A_\ell$ (see \cite[Lemma 10.2(ii)]{[2]}). Now construct explicitly $d$ morphisms $\varphi_k\in\operatorname{Hom}_{\mathbb{C}G}(M,\mathcal{H})$ (one for each $k=1,\dots,d$) as follows:

\[\varphi_k : M \longrightarrow \mathcal{H},\qquad \varphi_k(e_{g}) := g\cdot r_k.\]

Note that for $g,g'\in G$ we have that

\[\varphi_k(g'\cdot e_g) = \varphi_k(e_{g'g}) = (g'g)\cdot r_k = g'\cdot(g\cdot r_k) = g'\cdot \varphi_k(e_g).\]

Thus $\varphi_k\in\operatorname{Hom}_{\mathbb{C}G}(M,\mathcal{H})$. Therefore, $\varphi_1,\dots,\varphi_d$ span the space $\operatorname{Hom}_{\mathbb{C}G}(M,\mathcal{H})$. Passing back to the tensor side, the image corresponds to an element $u_k\in (\mathcal{H}\otimes M^*)^G$:
\[u_k = \sum_{j=1}^d A_{kj}\otimes y_j, \quad where \ \ \ A_{kj}\in \mathcal{H}.\]

By definition of the correspondence,
\[A_{kj}=\varphi_k(e_{g_j})=g_j\cdot r_k.\]

So we obtain that:

\begin{equation}\label{eq1b}
A_{kj} = g_j\cdot r_k,\quad 1\le k,j\le d.
\end{equation}

Thus, in this context, the matrix $E$ in Remark \ref{obs}(b) coincides exactly with the matrix $A_{kj}$ in (\ref{eq1b}) (see also \cite[Def. 10.6]{[2]}). Recall that $\mathcal{A}$ is the set of hyperplanes $Fix \ g$, where $g$ is a reflection in $G$. Gutkin's Theorem \cite[Th. 10.13]{[2]} states that there exists a constant $c\ne0$ such that

\[det(E)=c \cdot \prod_{Fix \hspace{0.05cm} g\in\mathcal A}L_{g}^{\,C(Fix \hspace{0.05cm} g,M)}.\]

Here $C(Fix \hspace{0.05cm} g,M)$ is determined by the restriction of $M$ to the cyclic stabilizer $G_{Fix \hspace{0.05cm} g}$. For the regular representation, one finds that:

\[C(Fix \hspace{0.05cm} g,M)=|G|\cdot \left( \frac{ord(g)-1}{2}\right) ,\]

\noindent (see for instance the proof of \cite[Th. 10.13]{[2]}), this completes the proof.\end{proof}
	
	\begin{corollary}\label{matrix2} Let $f:(\mathbb{C}^n,0)\longrightarrow (\mathbb{C}^{n+1},0)$, $f(\textbf{x})=(w(\textbf{x}),h(\textbf{x}))$, be a reflected graph map germ. Then
		
		\begin{equation*}
			\lambda[\textbf{X},Z]=w^{\ast}\left( \dfrac{1}{\prod_{i=1}^{\ell}{L_{H_i}}^{|G|(e_i-1)/2}}  E \cdot A_Z \cdot Adj(E) \right)
		\end{equation*}
		
		\noindent where $Adj(E)$ denotes the adjoint matrix of $E$ (the transpose of the cofactor matrix of $E$).
		
	\end{corollary}
	
	\begin{example}\label{z2} Consider $G=\mathbb{Z}_2 \times \mathbb{Z}_2$ and a double fold map germ $f:(\mathbb{C}^2,0)\rightarrow (\mathbb{C}^{3},0)$ given by $f(x,y)=(x^2,y^2, h(x,y))$. We can write $h(x,y)=xp_1+yp_2+xyp_3$, where $p_i=p_i(x^2,y^2)$. Note that $r_1=1$, $r_2=x$, $r_3=y$ and $r_4=xy$ generates $\mathbb{C}[x,y]/(x^2,y^2)$ as a $\mathbb{C}$-complex vector space. The matrices $E$, $A_Z$ and $Adj(E)$ are
		
		$$E=\begin{bmatrix}
			1 & 1 & 1 & 1 \\
			x & -x & x & -x \\
			y & y & -y & -y \\
			xy & -xy & -xy & xy \
		\end{bmatrix} \hspace{0.3cm} \textrm{,} \hspace{0.3cm} Adj(E)=\begin{bmatrix}
			4x^2y^2 & 4xy^2 & 4x^2y & 4xy \\
			4x^2y^2 & -4xy^2 & 4x^2y & -4xy \\
			4x^2y^2 & 4xy^2 & -4x^2y &-4xy \\
			4x^2y^2 & -4xy^2 & -4x^2y & 4xy
		\end{bmatrix},$$
		
		and
		
		$$A_Z=\begin{bmatrix}
			h_1-Z & 0 & 0 & 0 \\
			0 & h_2-Z & 0 & 0 \\
			0 & 0 & h_3-Z & 0 \\
			0 & 0 & 0 & h_4-Z
		\end{bmatrix},$$
		
		\noindent where $h_1=xp_1+yp_2+xyp_3$, $h_2=-xp_1+yp_2-xyp_3$, $h_3=xp_1-yp_2-xyp_3$ and $h_4=-xp_1-yp_2+xyp_3$.\\
		
		The group $\mathbb{Z}_2 \times \mathbb{Z}_2$ has two reflecting hyperplanes given by $L_{g_1}=\textbf{V}(x)$ and $L_{g_2}=\textbf{V}(y)$. All reflections of G have order 2, therefore $e_1=e_2=2$. By Proposition \ref{det} we obtain that $det(E)=cx^2y^2$ for some constant $c$. Actually, $det(E)=16x^2y^2$. 
		
		Making the product $E \cdot A_Z \cdot Adj(E)$ and multiplying the entries of the resulting matrix by $1/16x^2y^2$ we obtain the following matrix
		
		$$\lambda[w,Z]=\begin{bmatrix}
			-Z & p_1 & p_2 & p_3 \\
			x^2p_1 & -Z & x^2p_3 & p_2 \\
			y^2p_2 & y^2p_3 & -Z & p_1 \\
			x^2y^2p_3 & y^2p_2 & x^2p_1 & -Z \
		\end{bmatrix}.$$
		
		\noindent Now, making the change of $x^2\mapsto X$, $y^2\mapsto Y$ in $\lambda[w,Z]$ we obtain the presentation matrix
		
		$$\lambda[X,Y,Z]=\begin{bmatrix}
			-Z & p_1 & p_2 & p_3 \\
			Xp_1 & -Z & Xp_3 & p_2 \\
			Yp_2 & Yp_3 & -Z & p_1 \\
			XYp_3 & Yp_2 & Xp_1 & -Z \
		\end{bmatrix},$$
		
		\noindent which is exactly the same matrix $\lambda[X,Y,Z]$ that appears in the proof of (\cite[Prop. 3.1]{[3]}).
	\end{example}
	
	\section{The image of a reflected graph map}\label{sectionimage}
	
	$ \ \ \ \ $ In this section, we provide an answer to Question $1$ in Introduction. Specifically, we give a a defining equation of the image of a reflected graph map germ $f$ from $(\mathbb{C}^n,0)$ to $(\mathbb{C}^{n+1},0)$ in terms of the action of the reflection group $G$ on $h$ (the last coordinate function of $f$). Let us start with an example to illustrate our answer to Question 1. \\ 
	
	Consider the reflection group $G=\mathbb{Z}_4$ acting on $\mathbb{C}^2 \times \mathbb{C}$ (trivially on the second factor) with orbit map $w=(x,y^4,z)$. Let $f:(\mathbb{C}^2,0) \longrightarrow (\mathbb{C}^3,0)$ be a reflected graph map germ, given by
	\begin{equation*}
		f(x,y)=(x,y^4,yp_1+y^2p_2+y^3p_3),
	\end{equation*} 
	where $p_i=p_i(x,y^4)$. Note that $r_1=1$, $r_2=y$, $r_3=y^2$ and $r_4=y^3$ generates $\mathbb{C}[x,y]/(x,y^4)$ as a complex vector space. Applying Theorem \ref{c3} we obtain a presentation matrix of $f_*({\cal O}_2)$ as a ${\cal O}_3$-module via $f$ given by
	$$\lambda[X,Y,Z]=\begin{bmatrix}
		-Z & p_1 & p_2 & p_3 \\
		Yp_3 & -Z & p_1 & p_2  \\
		Yp_2 & Yp_3 & -Z & p_1 \\
		Yp_1 & Yp_2 & Yp_3 & -Z \
	\end{bmatrix}.$$.
	
	Therefore, calculating the determinant of $\lambda[X,Y,Z]$, a defining equation of the image of $f$ is given by
	\begin{center}
		$F(X,Y,Z)=Z^4+Q_2Z^2-Q_1Z+Q_0$,
	\end{center}
	
	where
	$$\begin{array}{rcl} 
		Q_2 & = & -4Yp_1p_3-2p_2^2, \\
		Q_1 & = & 4Yp_1^2p_2+4Y^2p_2p_3^2, \\  
		Q_0 & = & Y^2(p_2^4-4p_1p_2^2p_3+2p_1^2p_3^2)-Y^3p_3^4-Yp_1^4. 
	\end{array}$$
	
	Now, set $q_i(x,y):=w_{\ast}(Q_i)=Q_i(x,y^4)$. Thus we obtain that
	
	$$\begin{array}{rcl} 
		q_2(x,y) & = & -4y^4p_1p_3-2p_2^2, \\
		q_1(x,y) & = & 4y^4p_1^2p_2+4y^8p_2p_3^2, \\  
		q_0(x,y) & = & y^8(p_2^4-4p_1p_2^2p_3+2p_1^2p_3^2)-y^{12}p_3^4-y^4p_1^4, 
	\end{array}$$
	
	\noindent where now we have $p_i=p_i(x,y^4)$.\\ 
	
	A curious fact is that we can express each $q_i$ in terms of symmetric polynomials in the ``variables'' $h_i$, more precisely:
	$$\begin{array}{rcl} 
		q_2(x,y) & = & h_1h_2+h_1h_3+h_1h_4+h_2h_3+h_3h_4, \\
		q_1(x,y) & = & h_1h_2h_3+h_1h_2h_4+h_1h_3h_4+h_2h_3h_4, \\  
		q_0(x,y) & = & h_1h_2h_3h_4,
	\end{array}$$
	
	\noindent where, $h_i={g_i}_{\bullet}(xp_1+y^2p_2+y^3p_3)$, more precisely, $h_1=yp_1+y^2p_2+y^3p_3$, $h_2=-yp_1+y^2p_2-y^3p_3$, $h_3=iyp_1-y^2p_2-iy^3p_3$ and $h_4=-iyp_1-y^2p_2+iy^3p_3$.\\ 
	
	Note that in particular, the polynomials $q_i$ are invariant under the action of $G$. We will show in this section that this same behaviour happens for any reflected graph map germ.
	
	\begin{lemma}\label{LEMMA} Let $G=\{g_1=Id,g_2,...,g_d\}$ be a reflection group and $h$ be a function in $\mathbb{C}\lbrace \textbf{x} \rbrace$. Define $h_{i}$ by ${g_i}_{\bullet}h$, that is, $h_{i}$ is the action of $g_i$ on $h$. Then for all $k \in \{1,2,...,d\}$ the polynomial
		\begin{equation*}
			q_{d-k}:=\sum_{i_1<i_2<\cdots<i_d} h_{i_1}h_{i_2}\cdots h_{i_k}
		\end{equation*}
		is invariant under the action of $G$. In particular, $q_{d-k} \in \mathbb{C}[w_1,\cdots, w_d]$.
	\end{lemma}
	
	\begin{proof} Consider the ring $A:=\mathbb{C}[h_1,h_2,...,h_d]$ and let $Sym(A)$ be the set of symmetric polynomials in the ``variables'' $h_1,h_2,...,h_d$. Note that by definition $q_{d-k} \in Sym(A)$ for each $k$. For all $l \geq 1$ consider the power sum symmetric polynomial $m_l={h_1}^l+{h_2}^l+\cdots+{h_d}^l$. A well-known result states that any symmetric polynomial in $h_1, ..., h_d$ can be expressed as a polynomial expression with rational coefficients in the power sum symmetric polynomials $m_1,\cdots,m_d$ (see \cite[Th. 6.1]{lang}). In other words, the $m_1,m_2,...,m_d$ are the generators of the algebra of symmetric polynomials $Sym(A)$, that is,
		\begin{equation*}
			Sym(A)=\mathbb{C}[m_1,...,m_d].
		\end{equation*}
		
		Since $q_{d-k} \in Sym(A)$, there is $Q_{d-k} \in \mathbb{C}[m_1,m_2,...,m_d]$, such that $q_{d-k}=Q_{d-k}(m_1 ,m_2,...,m_d)$. From the fact that the Reynolds operator is a projection, it follows that each $m_i$ is invariant under the action of $G$. Therefore, $q_{d-k}$ is also invariant under the action of $G$.\end{proof}\\
	
	Now, since each $q_{d-k}$ is invariant under the action of $G$, i.e., we can think that $q_{d-k}$ is in $\mathbb{C}[w_1,\cdots, w_d]$. Therefore, we can define $w^{\ast}(q_{d-k}(w_1,\cdots,w_d))=:Q_{d-k}(X_1,\cdots,X_d)$, i.e, we can change the ``variable'' $w_i$ by the (target) variable $X_i$ in $q_{d-k}$. We are now able to present a defining equation of the image of a reflected graph map germ. 
	
	\begin{theorem}\label{teeeo} Let $G$ be a reflection group of order $d$. Let $f=(\textbf{w},h)$ be a reflected graph map germ from $(\mathbb{C}^n,0)$ to $(\mathbb{C}^{n+1},0)$. The image of $f$ is given as the zero set of $F$ (which is the determinant of the presentation matrix of $f_*{\cal O}_n$) where $F$ is described as the following alternating sum 
		\begin{equation}\label{imagem}
			F(\textbf{X},Z)= Z^d-Q_{d-1}Z^{d-1}+Q_{d-2}Z^{d-2}+\cdots+(-1)^{d-1}Q_1Z+(-1)^dQ_0,
		\end{equation}
		\noindent where $Q_{d-k}=q_{d-k}(\textbf{X})$, and $q_{d-k}$ is described in Lemma \ref{LEMMA}.
	\end{theorem}
	
	\begin{proof} Consider the image of $f$ given by the zero set of $F(\textit{\textbf{X}},Z)$, where $F$ denotes the determinant of the presentation matrix $\lambda[\textit{\textbf{X}},Z]$ of $f_*{\cal O}_n$ as a ${\cal O}_{n+1} $-module via $f$. Consider the pullback $F(\textit{\textbf{w}},Z)$ of $F(\textit{\textbf{X}},Z)$ by $w$. It follows by Lemma \ref{prim} that
		\begin{equation}
			F(\textit{\textbf{w}},Z)=\prod_{k=1}^{d}\left(Z-{g_k}_{\bullet}(h(\textit{\textbf{x}}))\right).
		\end{equation}
		
		\noindent Considering $h_k:={g_k}_{\bullet} h$ as in Lemma \ref{LEMMA}, we have that
		
		\begin{equation}\label{exp1}
			F(\textit{\textbf{w}},Z) = (Z-h_1)(Z-h_2)\cdots(Z-h_d).
		\end{equation}	
		
		\noindent Expanding the expression (\ref{exp1}), it follows from Lemma \ref{LEMMA} that 
		
		\begin{equation*}
			F(\textit{\textbf{w}},Z)= Z^d-q_{d-1}Z^{d-1}+q_{d-2}Z^{d-2}+\cdots+(-1)^{d-1}q_1Z+(-1)^dq_0.
		\end{equation*} 
	Since the polynomials $q_{d-k}$ are invariant, we can consider $w^{\ast}(q_{d-k}(w_1,\cdots,w_d)):=Q_{d-k}(X_1,\cdots,X_d)$, therefore  
		\begin{equation}\label{imagec}
			F(\textit{\textbf{X}},Z)= Z^d-Q_{d-1}Z^{d-1}+Q_{d-2}Z^{d-2}+\cdots+(-1)^{d-1}Q_1Z+(-1)^dQ_0.
		\end{equation}
		as desired.\end{proof}
	
	\section{The multiplicity of a reflection map germ}
	
	$ \ \ \ \ $  In this section, we will present an upper (and also a lower) bound for the multiplicity of the image of a generically $1$-to-$1$ reflection map germ $f$ from $(\mathbb{C}^n,0)$ to $(\mathbb{C}^{n+1},0)$, in general, not necessarily a reflected graph map. To state our result, we begin by recalling the notion of multiplicity. \
	
	Consider a germ of analytic function $F:(\mathbb{C}^{n+1},0)\longrightarrow (\mathbb{C},0)$ reduced at the origin with $F \neq 0$. Let $(\textbf{V}(F),0)$ be the germ of the zero set of $F$ at the origin. Write
	\begin{equation*}
		F=F_m+F_{m+1}+\cdots+F_k+\cdots
	\end{equation*}
	where each $F_k$ is a homogeneous polynomial of degree $k$ and $F_m \neq 0$. The integer $m$ is called the multiplicity of $\textbf{V}(F)$ in 0 and is denoted by $m(\textbf{V}(F),0)$. Clearly the multiplicity of $m(\textbf{V}(F),0)$ is greater than or equal to $1$. An important property of the multiplicity is that $m(\textbf{V}(F),0)=1$ if and only if $(\textbf{V}(F),0)$ is non singular.
	
	Once we have the notion of multiplicity in hand, let us return to the Question $3$ in the Introduction. Consider a reflection map germ $f:(\mathbb{C}^n,0)\rightarrow (\mathbb{C}^{n+1},0)$. If $f=(\textit{\textbf{w}},h)$ is a singular generically $1$-to-$1$ reflected graph map germ, then by Theorem \ref{teeeo} we have that 
	
	\begin{equation}\label{eq2}
		F(\textit{\textbf{X}},Z)= Z^d-Q_{d-1}Z^{d-1}+Q_{d-2}Z^{d-2}+\cdots+(-1)^{d-1}Q_1Z+(-1)^dQ_0
	\end{equation}
	
	\noindent is a defining equation of the image of $f$, i.e, $f(\mathbb{C}^n)=\textbf{V}(F)$. Since $f$ is generically $1$-to-$1$, we have that $F$ is reduced (see \cite[Prop. 3.1]{[12]}). As a consequence, we obtain that
	
	\begin{equation}\label{eq3}
		2 \leq m(f(\mathbb{C}^n),0)\leq d=|G|.
	\end{equation}
	
	However, if $f$ is a reflection map germ (not necessarily a reflected graph map germ one), it is not clear how big the multiplicity of the image can be. In Theorem \ref{teo5.2}, we provide an upper bound (and also a lower bound) that generalizes the one given in (\ref{eq3}). Before stating the theorem, we introduce an auxiliary lemma. 

\begin{remark}\label{remark1}	
	Let $(X,0)$ and $(Y,0)$ be germs of irreducible analytic sets. Let $f:(X,0)\rightarrow (Y,0)$ be a finite surjective analytic map germ. We denote the degree of a map $f$ by $deg(f)$. Roughly speaking, the degree of $f$ is the number of pre-images of a generic value in the image of $f$. The precise definition is given for instance in \cite[Def. D.2]{[7]}. An important fact about the concept of the degree of a map is its multiplicative property. Suppose $g:(Y,0)\rightarrow(W,0)$ is a finite surjective analytic map germ, where $(W,0)$ is an irreducible analytic set, then $deg(g\circ f)=deg(g)\cdot deg(f)$.
	
 Suppose that $(X,0)$ is of dimension $n$ and $(X,0)\subset (\mathbb{C}^{n+1},0)$. Let $l_1,l_2,\cdots, l_n:(\mathbb{C}^{n+1},0)\rightarrow (\mathbb{C},0)$ be generic (reduced) linear forms. Let $\pi:(X,0)\rightarrow (\mathbb{C}^n,0)$, $\pi=(l_1,\cdots, l_n)$, be the restriction to $X$ of the (generic) linear projection $\pi$ from $\mathbb{C}^{n+1}$ to $\mathbb{C}^n$. Here, ``\textit{generic linear projection}'' means that $Ker(\pi):= \pi^{-1}(0)$ is a generic line in $\mathbb{C}^{n+1}$ such that $Ker(\pi) \cap X = \{0\}$. 
 
  For a generic $x$ close enough to $0$, $\pi^{-1}(x)$ is a subspace parallel to $Ker(\pi)$ which intersects $X$ in a finite number of points; this number is precisely $m(X,0)$ (see for instance \cite[Sec. D.3]{[7]}). In other words, the multiplicity can be seen as the local intersection number at $0$ of $X$ with a generic line in $\mathbb{C}^{n+1}$. We note that this local intersection number is independent of the choice of the generic line (see \cite[Sec. D.3]{[7]}).
\end{remark}

	\begin{lemma}\label{lemaaux} Consider a reflection map germ $f:(\mathbb{C}^n,0)\rightarrow (\mathbb{C}^{n+1},0)$, $f=w \circ \rho$. As in Remark \ref{remark1}, let $\pi=( l_1,...,l_n)$ be a generic linear projection from $(\mathbb{C}^{n+1},0)$ to $(\mathbb{C}^n,0)$. Consider the image of $f$ with the Fitting structure, that is, $(f(\mathbb{C}^n),0)=(\textbf{V}({\cal F}_0(f_*{ \cal O}_n)),0)$ and $L:(\mathbb{C}^{n+1},0)\rightarrow (\mathbb{C},0)$ be a reduced (non singular) analytic function such that $(\rho(\mathbb{C}^n),0)=(\textbf{V}(L),0)$. If $f$ is generically $1$-to-$1$ then
		
		\begin{equation*}
			m(f(\mathbb{C}^n),0)=\dim_{\mathbb{C}} \dfrac{{\cal O}_{n+1}}{\left\langle L, l_1 \circ w,...,l_n \circ w \right\rangle }.
		\end{equation*}
	\end{lemma}
	
	\begin{proof} As in Section \ref{secpres}, set $(f(\mathbb{C}^n),0)=(\mathcal{X},0)$ and $(\rho(\mathbb{C}^n),0)=(\mathcal{Y},0)$.  Let $\pi: (\mathbb{C}^{n+1},0) \longrightarrow (\mathbb{C}^n,0)$ be a generic linear projection,  $\pi(\textit{\textbf{x}},x_{n+1})=(l_1(\textit{\textbf{x}},x_{n+1}),...,l_n(\textit{\textbf{x}},x_{n+1}))$. By Remark \ref{remark1} we obtain that:
		
		\begin{center}
			$deg(\pi \circ w_{\mid_{\cal Y}})=deg(w_{\mid_{\cal Y}})\cdot deg (\pi_{\mid_{\cal X}})$ $ \ \ \ $ and $ \ \ \ $ $m(f(\mathbb{C}^n),0)=deg(\pi_{\mid_{\cal X}})$. 
		\end{center}
		
		\noindent Furthermore,
		\begin{equation*}
			1=deg(f)=deg(\rho \circ w_{\mid_{\cal Y}})=deg(\rho)\cdot deg (w_{\mid_{\cal Y}})=deg (w_{\mid_{\cal Y}}).
		\end{equation*}
		
		\noindent Hence, $deg(\pi \circ w_{\mid_{\cal Y}})= deg (\pi_{\mid_{\cal X}})$. Therefore, $m(f(\mathbb{C}^n)_{red})=deg(\pi \circ w_{\mid_{\cal Y}})$. Finally, we have that
		\begin{equation}
			deg(\pi \circ w_{\mid_{\cal Y}})=\dim_{\mathbb{C}} \dfrac{{\cal O}_{\mathcal{Y},0}}{\left\langle l_1(x),...,l_n(w)\right\rangle}=\dim_{\mathbb{C}} \dfrac{{\cal O}_{n+1}}{\left\langle L, l_1 \circ w,...,l_n \circ w \right\rangle} .
		\end{equation}
		
		\noindent which concludes the proof.\end{proof}\\
	
	Let $(\mathcal{Y}_1,0),\cdots, (\mathcal{Y}_{n+1},0)$ be germs of hypersurfaces in $(\mathbb{C}^{n+1},0)$. Let $L_1,\cdots, L_{n+1}$ be germs of analytic functions from $(\mathbb{C}^{n+1},0)$ to $(\mathbb{C},0)$ such that $(\mathcal{Y}_i,0)=(\textbf{V}(L_i),0)$. We will denote the intersection multiplicity of $(\mathcal{Y}_1,0),\cdots, (\mathcal{Y}_{n+1},0)$ at $0$ by $i(\mathcal{Y}_1,\cdots, \mathcal{Y}_{n+1})$. If the intersection $(\mathcal{Y}_1,0)\cap \cdots \cap (\mathcal{Y}_{n+1},0)$ is just the origin, then the intersection multiplicity is a finite number and can be calculated as
	
	\begin{center}
		$i(\mathcal{Y}_1,\cdots, \mathcal{Y}_{n+1}) = \dim_{\mathbb{C}} \dfrac{{\cal O}_{n+1}}{\left\langle L_1, L_2 ,...,L_{n+1} \right\rangle}\textcolor{blue}{.} $
	\end{center}
	
	\begin{remark}\label{remarkintersection} We remark that the intersection multiplicity of hypersurfaces is greater than or equal to the product of the multiplicities of each hypersurface, i.e.
		
		\begin{equation}\label{intersection}
			i(\mathcal{Y}_1,\cdots, \mathcal{Y}_{n+1}) \geq m(\mathcal{Y}_1,0) \cdot m(\mathcal{Y}_2,0) \cdots m(\mathcal{Y}_{n+1},0). 
		\end{equation}
		
		\noindent with equality if and only if the intersection $(\mathcal{Y}_1,0)\cap \cdots \cap (\mathcal{Y}_{n+1},0)$ is transversal, i.e., the intersection of tangent cones of $(\mathcal{Y}_j,0)$ is also just the origin $0$ in $\mathbb{C}^{n+1}$. See \cite[Ch. 7]{fulton} and \cite[p. 151]{chirka} for details on the intersection multiplicity of hypersurfaces.
		
	\end{remark}
	
	\begin{theorem}\label{teo5.2} Let $f:(\mathbb{C}^n,0)\longrightarrow (\mathbb{C}^{n+1},0)$, $f=w \circ \rho$, be a generically $1$-to-$1$ reflection map germ, $G$ be a reflection group acting on $\mathbb{C} ^{n+1}$. Let $d_1\leq d_2 \leq  \cdots \leq d_{n+1}$ be the degrees of $G$. Then,
		\begin{equation}\label{eq19}
			d_1 d_2 \cdot ... \cdot d_n \leq m(f(\mathbb{C}^n),0) \leq d_2 d_3 \cdot ... \cdot d_{n+1} \leq |G|.
		\end{equation}
		
	\end{theorem}
	
	\begin{proof}Denote by $(\textit{\textbf{X}})=(X_1,\cdots, X_{n+1})$ the coordinates of $\mathbb{C}^{n+1}$ (the target of $f$).
	
	 Let $L:(\mathbb{C}^{n+1},0)\rightarrow (\mathbb{C},0)$ be a (reduced) non singular analytic function such that $(\rho(\mathbb{C}^n),0)=(\textbf{V}(L),0)$. We have that $\rho:(\mathbb{C}^n,0)\rightarrow (\mathbb{C}^{n+1},0)$ and $w:(\mathbb{C}^{n+1},0)\rightarrow (\mathbb{C}^{n+1},0)$, with $w=(w_1,\cdots,w_{n+1})$. In order to avoid confusion with the notation, let $(\textit{\textbf{x}})=(x_1,...,x_{n})$ be the coordinates of $\mathbb{C}^n$ (the source of $\rho$), $(\textit{\textbf{y}})=(y_1,...,y_{n+1})$ be the coordinates of $\mathbb{C}^{n+1}$ (the source of $w$, equivalently in the target of $\rho$). Let $\pi: (\mathbb{C}^{n+1},0) \longrightarrow (\mathbb{C}^n,0)$ be a generic linear projection, 
	
\begin{center}
	$\pi(\textit{\textbf{X}})=(l_1(\textit{\textbf{X}}),...,l_n(\textit{\textbf{X}}))$, 
\end{center}
	
\noindent in the sense of Remark \ref{remark1}, where for $i=1,\cdots,n$, $l_{i}(\textit{\textbf{X}}):=b_{i,1}X_1+b_{i,2}X_2+\cdots+b_{i,n+1}X_{n+1}$,  $b_{i,j}\in \mathbb{C}$, and 

\begin{center}
$Ker(\pi) \cap f(\mathbb{C}^n)=\lbrace 0 \rbrace$. 
\end{center}

Thus, $l_i \circ w=\sum_{j=1}^{n+1} b_{i,j}w_j$. By the genericity of $\pi$ we can assume that all the $b_{i,j}$ are non-zero complex constants. By Lemma \ref{lemaaux} we have that 	
\begin{equation}\label{eq10}
m(f(\mathbb{C}^n),0) =  dim_{\mathbb{C}} \dfrac{{\cal O}_{n+1}}{\left\langle L, \ \sum_{j=1}^{n+1} b_{1,j}w_j, \  \sum_{j=1}^{n+1} b_{2,j}w_j, \ ... \ , \ \sum_{j=1}^{n+1}b_{n,j}w_j \right\rangle}.
\end{equation}
		
After equivalences in the quotient ring in (\ref{eq10}) and by the genericity of $\pi$, we can find non-zero complex constants $\widetilde{b_{i,j}}$ such that
\begin{equation}\label{eq11}
m(f(\mathbb{C}^n),0) =  dim_{\mathbb{C}} \dfrac{{\cal O}_{n+1}}{\left\langle L, \ \sum_{j=1}^{n+1} \widetilde{b_{1,j}}w_j, \  \sum_{j=2}^{n+1} \widetilde{b_{2,j}}w_j, \  \sum_{j=3}^{n+1} \widetilde{b_{3,j}}w_j, \ ... \ , \ \sum_{j=n}^{n+1}\widetilde{b_{n,j}}w_j \right\rangle}.
\end{equation}

Note that the right-hand side of (\ref{eq11}) can be viewed as an intersection multiplicity of $n+1$ hypersurfaces in $\mathbb{C}^{n+1}$. Using the notation $\mathcal{Y}=V(L)$, we obtain that	
\begin{equation}\label{eq12}
m(f(\mathbb{C}^n),0) = i\left( \mathcal{Y} ,\textbf{V}\left( \sum_{j=1}^{n+1} \widetilde{b_{1,j}}w_j \right), \ \textbf{V}\left( \sum_{j=2}^{n+1} \widetilde{b_{2,j}}w_j \right), \ ... \ , \ \textbf{V}\left( \sum_{j=n}^{n+1} \widetilde{b_{n,j}}w_j \right)\right).
\end{equation}

Note that $d_i=m(\textbf{V}(w_i),0)$ by definition and $m(\mathcal{Y},0)=1$. By (\ref{intersection}) in Remark \ref{remarkintersection} and (\ref{eq12}) we conclude that
\begin{equation*}
m(f(\mathbb{C}^n),0) \geq  d_1d_2 \cdot ... \cdot d_n.
\end{equation*}		
		Now, we need to obtain an upper bound for the multiplicity of the image of $f$. Since $w$ if finite, then $\mathcal{O}_{n+1}/\langle w_1(\textit{\textbf{y}}), \cdots, w_{n+1}(\textit{\textbf{y}}) \rangle < \infty$. Therefore, if $y_i$ is a factor of $w_j$, then $y_i$ is not a factor of $w_s$ for all $s \neq j$. After a change of coordinates in $(\mathbb{C}^{n+1},0)$, we can assume without loss of generality that $y_{n+1}$ does not divide $w_2,\cdots, w_{n+1}$. Since $\rho$ is an embedding, we can write $\rho$ as		
		\begin{center}
		$\rho(\textit{\textbf{x}})=(\hat{l}_1(\textit{\textbf{x}})+g_1(\textit{\textbf{x}}), \ \hat{l}_2(\textit{\textbf{x}})+g_2(\textit{\textbf{x}}), \ \cdots, \ \hat{l}_{n+1}(\textit{\textbf{x}})+g_{n+1}(\textit{\textbf{x}}))$,
		\end{center}		
\noindent where $g_i \in \textbf{m}^2$ for all $i=1,\cdots, n+1$, $\textbf{m}$ denotes the maximal ideal of $\mathbb{C}\lbrace \textit{\textbf{x}} \rbrace$ and $\hat{l}_i(\textbf{x})=a_{i,1}x_1+\cdots+a_{i,n}x_n$. Note that the $n+1$ vectors $v_j:=(a_{1,j},\cdots,a_{n+1,j})$ in $\mathbb{C}^{n+1}$ (with $j=1,\cdots, n+1$) generate the tangent space of $(\mathcal{Y},0)$ (actually, only $n$ vectors are needed). Up to a change of coordinates in $\mathbb{C}^{n+1}$ we can assume that
	\begin{equation}\label{eq14}
\rho(\textit{\textbf{x}})=(x_1+g_1(\textit{\textbf{x}}), \ x_2+g_2(\textit{\textbf{x}}), \ \cdots, \ x_n+g_n(\textit{\textbf{x}}), \ \hat{l}_{n+1}(\textit{\textbf{x}})+ g_{n+1}(\textit{\textbf{x}})),	
	\end{equation}
\noindent and therefore,
	\begin{equation}\label{eq13}
	f(\textit{\textbf{x}})=(w_1(x_1+g_1(\textit{\textbf{x}}),\cdots,\hat{l}_{n+1}(\textit{\textbf{x}})+ g_{n+1}(\textit{\textbf{x}})), \cdots, \ w_{n+1}(x_1+g_1(\textit{\textbf{x}}),\cdots, \hat{l}_{n+1}(\textit{\textbf{x}})+ g_{n+1}(\textit{\textbf{x}}))).
	\end{equation}
		
Note that in (\ref{eq14}) $y_{n+1}$ (the last coordinate of $\mathbb{C}^{n+1}$) is not affected by the last coordinate change. Therefore, we can still assume that $y_{n+1}$ does not divide $w_2,\cdots, w_{n+1}$. Let $\pi_1: (\mathbb{C}^{n+1},0) \longrightarrow (\mathbb{C}^n,0)$ be the linear projection defined by
	
\begin{center}
	$\pi_1(X_1,X_2,\cdots,X_{n+1})=(X_2,\cdots,X_{n+1})$. 
\end{center}
		
\noindent \textbf{Statement:} The map germ $(\pi_1 \circ f):(\mathbb{C}^n,0)\rightarrow (\mathbb{C}^n,0)$ is finite.\\

By (\ref{eq13}) and \cite[Th. D.5]{[7]}, it is sufficient to show that 
\begin{equation}\label{eq15}
\dfrac{\mathbb{C}\lbrace x_1,\cdots,x_n \rbrace}{\langle \ w_2(x_1+g_1(\textit{\textbf{x}}),\cdots,\hat{l}_{n+1}(\textit{\textbf{x}})+ g_{n+1}(\textit{\textbf{x}})), \cdots, \ w_{n+1}(x_1+g_1(\textit{\textbf{x}}), \ \cdots, \hat{l}_{n+1}(\textit{\textbf{x}})+ g_{n+1}(\textit{\textbf{x}})) \ \rangle}		
	\end{equation}
is a finite-dimensional vector space over $\mathbb{C}$.

Note that the quotient ring in (\ref{eq15}) is isomorphic to the following quotient ring 
\begin{equation}\label{eq16}
\dfrac{\mathbb{C}\lbrace x_1,\cdots,x_n,u \rbrace}{\langle \ \hat{l}_{n+1}(\textit{\textbf{x}})+ g_{n+1}(\textit{\textbf{x}})-u, w_2(x_1,\cdots,x_n,u)+ \textbf{...} \ , \ \cdots, \ w_{n+1}(x_1,\cdots,x_n,u)+ \textbf{...}  \ \rangle}		
	\end{equation}
	
\noindent where $\textbf{...}$ appearing in $w_i(x_1,x_2,\cdots,x_n,u)+\textbf{...}$ consists of terms of degree greater than the degree of $w_i$. In other words, $w_i(x_1,x_2,\cdots,x_n,u)+\textbf{...}$ is just the expansion of $w_i(x_1+g_1(\textit{\textbf{x}}),\cdots,x_n+g_n(\textit{\textbf{x}}),u)$, and $u$ is a new variable. 

 Since the ring $\mathcal{O}_{n+1}/\langle w_2(\textit{\textbf{y}}),\cdots,w_{n+1}(\textit{\textbf{y}})\rangle$ is unmixed (it is a Cohen-Macaulay ring), we obtain that the zero-divisors of
\begin{equation}\label{eq17}
R:=\dfrac{\mathbb{C}\lbrace x_1,\cdots,x_n,u \rbrace}{\langle \ w_2(x_1,\cdots,x_n,u) \ , \ \cdots, \ w_{n+1}(x_1,\cdots,x_n,u)  \ \rangle}.		
	\end{equation}

\noindent are precisely the elements in the union of the minimal primes of $R$.

Since $y_{n+1}$ does not divide $w_i(y_1,\cdots,y_n,y_{n+1})$ for all $i\neq 1$, then $u$ does not divide $w_i(x_1,\cdots,x_n,u)$ for all $i\neq 1$. Thus $u$ cannot be contained in any minimal prime of $R$. It follows that $u$ is not a zero divisor of $R$. By a suitable version of Krull's Principal Ideal Theorem, we obtain that
\begin{equation}\label{eq18}
\dfrac{R}{\langle u \rangle}\simeq \dfrac{\mathbb{C}\lbrace x_1,\cdots,x_n,u \rbrace}{\langle \ u, w_2(x_1,\cdots,x_n,u) \ , \ \cdots, \ w_{n+1}(x_1,\cdots,x_n,u)  \ \rangle}.		
	\end{equation}
 
\noindent has dimension zero, hence it is a finite-dimensional vector space over $\mathbb{C}$. This implies that the intersection in (\ref{eq16}) (and also in (\ref{eq15})) is transversal. By Remark \ref{remarkintersection}, we conclude that the dimension as a $\mathbb{C}$-vector space of the quotient ring in (\ref{eq15}) is precisely $d_2\cdot d_3 \cdots d_{n+1}$, which proves the statement.\\

To conclude the proof, in the Mond-Pellikaan algorithm (see \cite[Sec. 2.2]{[12]}), consider the projection $\pi_1$ above. By the statement, we have that $\pi_1 \circ f$ is finite. Thus the presentation matrix $\lambda$ of $f_{\ast}(\mathcal{O}_n)$ with respect to $\pi_1$ has the following form:	

$$\lambda[\textit{X}_1,\cdots,\textit{X}_{n+1}]=\left[ \begin{array}{rrrrr}
			\alpha_{1,1}-X_1 & \alpha_{1,2} & \alpha_{1,3} & \cdots & \alpha_{1,d'} \\
			\alpha_{2,1} & \alpha_{2,2}-X_1 & \alpha_{2,3} & \cdots & \alpha_{2,d'} \\
			\alpha_{3,1} & \alpha_{3,2} & \alpha_{3,3}-X_1 & \cdots & \alpha_{3,d'} \\
			\vdots & \vdots & \vdots & \cdots & \vdots \\
			\alpha_{d',1} & \alpha_{d',2} & \alpha_{d',3} & \cdots & \alpha_{d',d'}-X_1
		\end{array}\right].$$
		
\noindent where $\alpha_{i,j} \in \mathbb{C}\lbrace X_2,\cdots, X_{n+1}\rbrace$ and $d'=d_2\cdot d_3 \cdots d_{n+1}$.

As a consequence, we obtain that the term $X_1^{d'}$ appears in the defining equation of $f(\mathbb{C}^n)$ given by the determinant of $\lambda$. Therefore, $m(f(\mathbb{C}^n),0)\leq d_2 \cdots d_{n+1}$. The last inequality in (\ref{eq19}) follows from the fact that $d_1\cdot d_2 \cdots d_{n+1}=|G|$, since $G$ is a reflection group (\cite{[6]}, see also \cite[Th. 4.19]{[2]}).\end{proof}
	
	\begin{remark}\label{cormult} Let $f:(\mathbb{C}^n,0)\longrightarrow (\mathbb{C}^{n+1},0)$ be a reflected graph map germ with respect to a reflection group $G$ acting on $\mathbb{C}^{n+1}=\mathbb{C }^{n} \times \mathbb{C}$, as in Section \ref{sectionimage}. Since $G$ acts trivially on the second factor, the orbit map $w:(\mathbb{C}^{n+1},0)\rightarrow (\mathbb{C}^{n+1},0)$ is $w(\textbf{y})=(w_1(\textbf{y}),\cdots,w_n(\textbf{y}),y_{n+1})$. In other words, if $d_1,\cdots,d_{n+1}$ are the degrees of $G$, then $d_{n+1}=1$. After a reordering of the remaining degrees of $G$, we can assume that $d_1\leq d_2 \leq \cdots \leq d_n$. By Theorem \ref{teo5.2} we obtain that
	
		\begin{equation*}
			d_1 d_2 \cdot ... \cdot d_{n-1} \leq m(f(\mathbb{C}^n),0) \leq d_1 \cdot ... \cdot d_n =d=|G|.
		\end{equation*}

When working with a reflected graph map, to avoid unnecessary notation, we can think of $G$ as acting nontrivially only on $\mathbb{C}^n$ (the first factor of $\mathbb{C}^n \times \mathbb{C}$, and not on all $\mathbb{C}^{n+1}$). In this case, we can consider only $d_1,\cdots, d_n$ as ``non-trivial degrees'' of $G$ and ignore the degree $d_{n+1}=1$. Whenever there is no danger of confusing the notation, we will use this simplification.				
		
	\end{remark}
	
	\section{Some applications and examples}
	
	$ \ \ \ \ $ In this section, we discuss present some applications of our results. In the first part, we provide an alternative proof of a result of (see \cite[Th. 5.2]{[rafael]}) about a defining equation of the double point space of a reflected graph map germ. In the second part, we will introduce the notion of a dihedral map germ and we will apply our results to describe the presentation matrix and a defining equation of the image of maps of this kind. Finally, in the last part we extend a result given by Marar and Nuño-Ballesteros (see \cite[Th. 3.4]{[3]}) about the non-existence of quasihomogeneous (with distinct weights) finitely determined reflection map germs, for some reflection groups.
	
	\subsection{Double point spaces}
	
	$ \ \ \ \ $ When we study a finite map $f$ from $(\mathbb{C}^n,0)$ to $(\mathbb{C}^{p},0)$ with $n\leq p$ the multiple point spaces of $f$ play an important role in the study of its geometry. 
	
	In the particular case where $p=n+1$ and $f$ is a reflection map, a natural question is if we can described a defining equation of the hypersurface of double points $D(f)$ in the source in terms of the action of $G$ in some sense. 
	
	Recently, Borges Zampiva, Peñafort-Sanchis, Oréfice Okamoto and Tomazella provided an answer to this question in a more general context. They give a defining equation of $D(f)$ which depends only on the action of $G$ in an appropriate way (see \cite[Th. 5.2]{[rafael]}). As our first application, we will use Lemma \ref{prim} to offer an alternative proof of the formula given in the case of reflected graph map germs. For the convenience of the reader we will recall the notion of the double point space of a map germ. We follow \cite{[8]} to describe the double point set of a map germ $f:(\mathbb{C}^n,0)\longrightarrow (\mathbb{C}^p,0)$, with $n<p$.\\
	
	Let $f:U \longrightarrow \mathbb{C}^p$ be a homomorphic map, where $U \subset \mathbb{C}^n$ is an open subset and $n<p$. Now let us define the set of double point space of $f$, denoted by $D^2(f)$. 
	Denote the diagonals of $\mathbb{C}^n \times \mathbb{C}^n$ and $\mathbb{C}^p \times \mathbb{C}^p$ by $\Delta_n$ and $\Delta_p$ and denote the ideal schemes of defining $\Delta_n$ and $\Delta_p$ by ${\cal I}_n$, ${\cal I}_p$. Write points of $\mathbb{C}^n \times \mathbb{C}^n$ as $(x,x')$. It is clear that for each $i=1,...,p$,
	\begin{equation*}
		f_i(x)-f_i(x') \in {\cal I}_n,
	\end{equation*}
	so there exist $\alpha_{i,j}$, $1 \leq i \leq n$, $1 \leq j \leq n+1$, such that 
	\begin{equation*}
		f_i(x)-f_i(x')=\sum \alpha_{i,j}(x-x')(x_j-x_j').
	\end{equation*}
	
	If $f(x)=f(x')$ and $x \neq x'$, then clearly every $n\times n$ minor of the matrix $\alpha=[\alpha_{i,j}]$ must vanish at $(x,x')$. Now denoting by ${\cal R}_n(\alpha)$ the ideal in ${\cal O}_{\mathbb{C}^{2n}}$ generated by the $n\times n$ minors of $\alpha$.
	
	\begin{definition}Let $f:U \longrightarrow \mathbb{C}^p$ be as above, then the double point space of $f$ is the complex space
	\end{definition}
	\begin{equation*}
		{ D}^2(f)=\textbf{V}((f\times f)^*{\cal I}_p+{\cal R}_n(\alpha)).
	\end{equation*}
	
%	Note that at a point $(x,x')$ that is not on the diagonal $\Delta_n$, ${\cal I}^2(f)$ is generated by the $f_i(x)-f_i(x')$ functions. Furthermore, the restriction of ${\cal I}^2(f)$ to diagonal $\Delta_n$ is the ideal generated by \textcolor{blue}{the $n\times n$ minors} of the Jacobian matrix of $f$, so $\Delta_n \cap {D}^2(f)$ is the singular set of $f$. 
	
	\begin{definition}\label{D2f} Let $f:(\mathbb{C}^n,0) \longrightarrow (\mathbb{C}^p,0)$ be a finite map germ, where $n \leq p$.\\ 
		
		\noindent (a) Let us denote by $I_p$ and $R_n(\alpha)$ the stalks at $0$ of the sheaves ${\cal I}_p$ and ${\cal R}_n(\alpha)$. Taking a representative of $f$ we define the double point space of the map germ $f$ as the analytic space
		\begin{equation*}
			{D}^2(f):=\textbf{V}((f\times f)^*I_p+R_n(\alpha)).
		\end{equation*}

		\noindent (b) The mapping $\pi:(D^2(f),0) \longrightarrow (\mathbb{C}^n,0)$, given by $(x,x') \longmapsto x$, is finite. The source double point set ${D}(f)$ is defined as the image of $\pi$ with the analytic structure given by the $0$-Fitting ideal of $\pi_* {\cal O}_{{D}^2(f)}$ , that is, 
		\begin{equation*}
			{D}(f):=\textbf{V}({\cal F}_0(\pi_* {\cal O}_{{*D}^2(f)})).
		\end{equation*}
	\end{definition}
	
	Now we will present an alternative proof for \cite[Th. 5.2]{[rafael]} in the case of a reflected graph map germ. We recall that $\cal R$ denotes the set of all reflections of $G$ and $\Delta_{g_k}$ denotes the Demazure operator (see Section \ref{prel}).
	
	\begin{proposition}\rm(\cite{[rafael]})\label{princi} \textit{Let $f:(\mathbb{C}^n,0) \longrightarrow (\mathbb{C}^{n+1},0)$, $f(\textbf{x})=(w(\textbf{x}),h(\textbf{x}))$ be a reflected graph map germ, then}
		\begin{equation}\label{teoo}
			{D}(f)=\textbf{V}\left( \left( \prod_{g_k \in {\cal R}}  \Delta_{g_k}(h)\right) \left( \prod_{g_k \notin {\cal R}, \ g_k \neq Id}\left( h-{g_k}_{\bullet}h \right) \right) \right).
		\end{equation}
	\end{proposition}	
	\begin{proof} Denote by $g_1=Id,g_2,\cdots, g_d$ the elements of $G$. By Theorem \ref{teeeo} we obtain that a defining equation of the determinant of the image of $f$ is given by		
		\begin{equation}\label{equa}
			F(\textit{\textbf{X}},Z)=Z^d-Q_{d-1}Z^{d-1}+Q_{d-2}Z^{d-2}+\cdots+(-1)^{d-1}Q_1Z+(-1)^dQ_0,
		\end{equation}	
		Note that replacing $\textit{\textbf{X}}$ by $\textit{\textbf{w}}$ in (\ref{equa}) and taking the partial derivative with respect to the variable $Z$ is equivalent to taking the partial derivative of (\ref{equa}) with respect to variable $Z$ and then replacing $\textit{\textbf{X}}$ with $\textit{\textbf{w}}$, i.e., the order of execution of these operations does not matter. By Lemma \ref{prim} we have that
		\begin{equation*}
			F(\textit{\textbf{w}},Z)=\prod_{k=1}^{d}\left(Z-{g_k}_{\bullet}(h(\textit{\textbf{x}}))\right).
		\end{equation*}
		\noindent Therefore,
		\begin{equation*}
			\dfrac{\partial F(\textit{\textbf{w}},Z)}{\partial Z}= \dfrac{ \partial \left(Z-{g_1}_{\bullet}{h(\textit{\textbf{x}})}\right)}{\partial Z} \left[\prod_{k=2}^{d}\left(Z-{g_k}_{\bullet}(h(\textit{\textbf{x}}))\right)\right] + \left( Z-{g_1}_{\bullet}{h(\textit{\textbf{x}})}\right) \dfrac{\partial \left[\prod_{k=2}^{d}\left(Z-{g_k}_{\bullet}(h(\textit{\textbf{x}}))\right)\right]}{\partial Z} ,
		\end{equation*}
		
		Note that $g_{1_{\bullet}}h(\textit{\textbf{x}})=h(\textit{\textbf{x}})$, replacing $Z$ by $h$ we obtain that
		
		\begin{equation}\label{h}
			\dfrac{\partial F(\textit{\textbf{w}},h)}{\partial Z}=(1-h(\textit{\textbf{x}}))\left[\prod_{k=2}^{d}\left(h(\textit{\textbf{x}})-{g_k}_{\bullet}(h(\textit{\textbf{x}}))\right)\right].
		\end{equation}
		
		By \cite{piene} (see also \cite[Prop. 2.5]{[1]}) we obtain that
		\begin{equation*}\label{for}
			(D(f),0)= \left( \textbf{V}\left(\dfrac{\prod_{k=2}^{d}\left( h(\textit{\textbf{x}})-{g_{k}}_{\bullet}h(\textit{\textbf{x}})\right) }{det(Jac(\textit{\textbf{w}}))}\right),0 \right).
		\end{equation*}
		
	Consider the hyperplanes $Fix g$, where $g$ runs through the reflections of $G$ and call these the reflecting hyperplanes of $G$. For each $i$, let $ord(g_i)$ be the order of the cyclic group $G_{Fix g_i}$ fixing $Fix g_i$ pointwise and let $L_{g_1}, L_{g_2}, . . . , L_{g_l}$ be linear forms such that $Fix g_i = Ker L_{g_i}$. We have that
		\begin{equation*}\label{for}
			\dfrac{\prod_{{g_k} \in {\cal R}}\left( h(\textit{\textbf{x}})-{g_{k}}_{\bullet}h(\textit{\textbf{x}})\right) }{det(Jac(\textit{\textbf{w}}))} \stackrel{Prop.\, \ref{li3}}{=} \dfrac{\prod_{{g_k} \in {\cal R}}\left( h(\textit{\textbf{x}})-{g_{k}}_{\bullet}h(\textit{\textbf{x}})\right)}{c\prod_{i=1}^{\ell} L_{g_i}^{{ord(g_i)}-1}}= c\prod_{{g_k} \in {\cal R}} \Delta_{g_k}(h)
		\end{equation*}
		
		\noindent where ${\cal R}$ denotes the set of all reflections of $G$ (see for instance \cite[Lemma 9.7]{[2]}). Therefore, 
		\begin{equation*}
			{D}(f)=\textbf{V}\left( \left( \prod_{g_k \in {\cal R}}  \Delta_{g_k}(h)\right) \left( \prod_{g_k \notin {\cal R}, \ g_k \neq Id}\left( h-{g_k}_{\bullet}h \right) \right) \right),
		\end{equation*} as which concludes the proof. \end{proof}

	\begin{example}\label{ex6.2} Consider $G=\mathbb{Z}_2\times \mathbb{Z}_2$ and $f(x,y)=(x^2,y^2,h(x,y))$, where 
		\begin{equation*}
			h(x,y)=xp_1(x^2,y^2)+yp_2(x^2,y^2)+xyp_3(x^2,y^2).
		\end{equation*}
		
		\noindent Note that the elements of $\mathbb{Z}_2\times \mathbb{Z}_2$ can be represented in $GL_2(\mathbb{C})$ by the following matrices 
		$$Id=\begin{bmatrix}
			1 & 0 \\
			0 & 1 \
		\end{bmatrix} \hspace{0.5cm}  \textrm{,} \hspace{0.5cm}  g_1=\begin{bmatrix}
			-1 & 0 \\
			0 & -1 \
		\end{bmatrix} \hspace{0.5cm}  \textrm{,} \hspace{0.5cm}  g_2=\begin{bmatrix}
			1 & 0 \\
			0 & -1 \
		\end{bmatrix}\hspace{0.5cm}  \textrm{e} \hspace{0.5cm}  g_3=\begin{bmatrix}
			-1 & 0 \\
			0 & 1 \
		\end{bmatrix}.$$
		
		\noindent Therefore, we obtain that
		$$\begin{array}{rcl} 
			h(x,y)-{g_1}_{\bullet}h(x,y) & = & 2(xp_1+yp_2). \\
			h(x,y)-{g_2}_{\bullet}h(x,y) & = & 2y(p_2+xp_3). \\  
			(x,y)-{g_3}_{\bullet}h(x,y) & = &2x(p_1+yp_3).
		\end{array}$$
		
		\noindent Note that $det(Jac(x^2,y^2))=4xy$. Thus, from Proposition \ref{princi}, we obtain that 
		\begin{equation}\label{eq1}
			{D}(f)=\textbf{V}((yp_3+p_1)(xp_3+p_2)(xp_1+yp_2)).
		\end{equation}
		
		\noindent We remark that the defining equation that appears in (\ref{eq1}) for ${D}(f)$ is exactly the same presented in \cite[Prop. 3.1]{[3]}.	
	\end{example}
	
	\begin{corollary}\label{non} With the notation used in Proposition \ref{princi} let $f=(\textbf{w},h)$ be a reflected graph map germ from $(\mathbb{C}^n,0)$ to $(\mathbb{C}^{n+1},0)$. Suppose that $h$ is regular, say 
		\begin{equation*}
			h(\textbf{x})=a_1x_1+a_2x_2+\cdots+a_nx_n+R(\textbf{x}),
		\end{equation*}	
		with $m(R(\textbf{x}),0)>1$ or $R=0$. Let $H$ be the hyperplane defined by $H:=\textbf{V}(\sum_{i=1}^{n}a_ix_i)$. Suppose that for every reflection $g \in G$, $H$ is not the reflecting hyperplane of $g$ (equivalently $H \cap (V(Jac(w))) \neq H$ as a set). Then,
		\begin{equation*}
			{D}(f)=\textbf{V}\left( \prod_{g_k \notin {\cal R}, \ g_k \neq Id}\left( h(\textbf{x})-{g_k}_{\bullet}h(\textbf{x}) \right) \right).
		\end{equation*}
	\end{corollary}
	
		\begin{proof} By Proposition \ref{princi} we have that
		\begin{equation}\label{eq20}
			{D}(f)=\textbf{V}\left( \left( \prod_{g_k \in {\cal R}}  \Delta_{g_k}(h(\textit{\textbf{x}}))\right) \left( \prod_{g_k \notin {\cal R}, \ g_k \neq Id}\left( h(\textit{\textbf{x}}))-{g_k}_{\bullet}h(\textit{\textbf{x}})) \right) \right) \right).
		\end{equation}
		
We just need to show that $\Delta_{g}(h(\textit{\textbf{x}}))$ is an invertible element in the ring of analytic series $\mathbb{C}\{\textbf{\textit{x}}\}$ for all $g \in \mathcal{R}$. If $g \in \mathcal{R}$ and $Fix\,g$ is its reflecting hyperplane, with $Fix\,g=Ker\,L_g$, then by the definition of the Demazure operator and Lemma \ref{li} we obtain that
\begin{equation}\label{eq21}
\Delta_{g}(h(\textit{\textbf{x}}))=\dfrac{h(\textit{\textbf{x}})-g_{\bullet}h(\textit{\textbf{x}})}{L_g}=Q(\textit{\textbf{x}}),
\end{equation}	
\noindent for some $Q(\textit{\textbf{x}}) \in \mathbb{C}\lbrace \textbf{\textit{\textbf{x}}}\rbrace$. 

The action of $G$ on $\mathbb{C}\lbrace \textbf{\textit{\textbf{x}}}\rbrace$ has the property that if $P \in \mathbb{C}\lbrace \textbf{\textit{\textbf{x}}} \rbrace$ is homogeneous of degree $d$, then $g_{\bullet}P$ is also homogeneous of degree $d$ (see \cite[Ch. 3, Sec. 2]{[2]}). Also, recall that $g_{\bullet}(P_1+P_2)=g_{\bullet}(P_1)+g_{\bullet}(P_2)$ for all $P_1,P_2 \in \mathbb{C}\lbrace \textbf{\textit{\textbf{x}}} \rbrace$. By hypothesis, $H$ is not a hyperplane of $g$. Therefore, 
\begin{equation*}
g_{\bullet}(a_1x_1+a_2x_2+\cdots+a_nx_n) \neq a_1x_1+a_2x_2+\cdots+a_nx_n,
\end{equation*}
\noindent i.e. the defining equation of $H$ is not invariant under the action of $g$. In this way, we can write

\begin{center}
$h(\textit{\textbf{x}})-g_{\bullet}h(\textit{\textbf{x}})=b_1x_1+b_2x_2+\cdots+b_nx_n+Q'(\textit{\textbf{x}})$,
\end{center}

\noindent where either $m(Q'(\textbf{\textit{x}}),0)>1$ or $Q'=0$. Also, $b_i \in \mathbb{C}$ for all $i$, with at least one $b_i$ being non-zero.

By (\ref{eq21}) we have that $L_g$ (which is a linear form) divides $h(\textit{\textbf{x}})-g_{\bullet}h(\textit{\textbf{x}})$. This implies that $Q(0)\neq 0$, and so $Q(\textit{\textbf{x}})$ must necessarily be an invertible element in $\mathbb{C}\{\textbf{\textit{x}}\}$.\end{proof}

	\begin{example} Let $f:(\mathbb{C}^2,0) \longrightarrow (\mathbb{C}^3,0)$ be a reflected graph map germ, given by 
		
		\begin{center}
			$f(x,y)=(x^2,y^2,x-3y+y^3)$. 
		\end{center}
		
		\noindent Using the notation in Example \ref{ex6.2}, by Corollary \ref{non} we obtain that 
		\begin{equation*}
			{D}(f)=\textbf{V}(h(x,y)-g_{1_{\bullet}}h(x,y)))= \textbf{V} \left(x-3y+y^3 \right).
		\end{equation*}
		
		Note that in this case we do not need to consider all elements of $G$ to obtain a defining equation of $D(f)$, only those that are not reflections (excluding of course the identity of the group).	
	\end{example}
	
	\subsection{Dihedral map germs}\label{dihedral}
	
	$ \ \ \ \ $ Inspired by the work of Marar and Nuño-Ballesteros in \cite{[3]} where they introduce the double fold map germs, in this section we will introduce the ``\textit{dihedral map germs}''. Consider the generators of the dihedral group ${D}_{2m}$ in $GL(\mathbb{C}^2)$ given by in Example \ref{example1},
	
	$$R=\left[ \begin{array}{rrrrr}
		0 & 1 \\
		1 & 0
	\end{array}\right] \hspace{0.5cm} \textrm{and} \hspace{0.5cm} S=\left[ \begin{array}{rrrrr}
		\zeta & 0 \\
		0 & \zeta^{m-1} 
	\end{array}\right].$$
	\noindent where $\zeta=e^{\frac{2\pi i}{m}}$.
	
	Note that with this representation $D_{2m}$ is a reflection group. By \cite[Prop. 2.22]{[2]} we have that the polynomials $w_1=x^m+y^m$ and $w_2=xy$ are algebraically independent and they are invariant by the action of $D_{2m}$ (in \cite{[2]} $D_{2m}$ is identified as being the group $G(m,m,2)$ in the Shephard and Todd classification).  Therefore the orbit map for the group $D_{2m}$ acting on $\mathbb{C}^2$ is $w(x,y)=(x^m+y^m, \  xy)$.\\
	%The group $D_{2m}$ has $m$ reflections, whose reflecting hyperplanes are  $L_{H_1}=x-y$, $L_{H_2}=x-\zeta y$, ..., $L_{H_{m-1}}=x-\zeta^{m-2} y$ and $L_{H_{m}}=x-\zeta^{m-1} y$.
	
	Thus, a the reflection map $f:(\mathbb{C}^2,0) \longrightarrow (\mathbb{C}^3,0)$ is a \textit{$2m$-dihedral map germ}, or simply ``\textit{a dihedral map germ}'', if $f$ is a reflected graph map germ in the following form
	\begin{equation*}
		f(x,y)=(x^m+y^m, \ xy, \ h(x,y)).
	\end{equation*} 
	
	\noindent Note that $r_1=1, r_2=x, r_3=x^2, \cdots, r_{m}=x^{m-1}$ and $r_{m+k}=y^k$ with $k \in \{1,...,m\}$ is a basis of $\mathcal{O}_2/(x^m+y^m,xy)$ as a $\mathbb{C}$-vector space. By taking coordinate changes and using the Malgrange preparation theorem we can write $h$ in the form
	
	\begin{equation*}
		h(x,y)=\sum_{j=1}^{m-1}x^jp_j+\sum_{k=1}^{m}y^kp_{m-1+k}
	\end{equation*}
	
	\noindent where $p_i=p_i(x^m+y^m,xy)$ and we suppose that $h(0,0)=0$. As an example, we will provide in the next result a study of a $6$-dihedral map germ, i.e., using the $D_6$ group.  We consider the following representation in $GL(\mathbb{C}^2)$ of the dihedral group $D_6$:
	$$g_1=Id=\begin{bmatrix}
		1 & 0 \\
		0 & 1 \
	\end{bmatrix},g_2=\begin{bmatrix}
		0 & 1 \\
		1 & 0 \
	\end{bmatrix}, g_3=\begin{bmatrix}
		\zeta & 0 \\
		0 & \zeta^2 \
	\end{bmatrix}, g_4=\begin{bmatrix}
		\zeta^2 & 0 \\
		0 & \zeta \
	\end{bmatrix}, g_5=\begin{bmatrix}
		0 & \zeta \\
		\zeta^2 & 0 \
	\end{bmatrix},g_6=\begin{bmatrix}
		0 & \zeta^2 \\
		\zeta & 0 \
	\end{bmatrix}.$$
	\noindent	where $\zeta=e^{\frac{2\pi i}{3}}$, so $\zeta^2+\zeta+1=0$.
	
	The orbit map for the group $D_6$ acting on $\mathbb{C}^2$ is $w(x,y)=(x^3+y^3,xy)$. Thus, a reflection map $f:(\mathbb{C}^2,0) \longrightarrow (\mathbb{C}^3,0)$ for the dihedral group $D_6$ is given by $f(x,y)=(x^3+y^3,xy,h(x,y))$, where $h(x,y)=xp_1+x^2p_2+yp_3+y^2p_4+y^3p_5$. Note that
	
	$$\begin{array}{rcl}
		h_1:= {g_1}_{\bullet}h & = & h.\\ 
		h_2:={g_2}_{\bullet}h & = & yp_1+y^2p_2+xp_3+x^2p_4+x^3p_5. \\  
		h_3:={g_3}_{\bullet}h & = & \zeta xp_1+\zeta^2x^2p_2+\zeta^2yp_3+\zeta y^2p_4+y^3p_5. \\
		h_4:={g_4}_{\bullet}h & = & \zeta^2xp_1+\zeta x^2p_2+\zeta yp_3+\zeta^2 y^2p_4+y^3p_5. \\
		h_5:={g_5}_{\bullet}h & = & \zeta yp_1+\zeta^2y^2p_2+\zeta^2xp_3+\zeta x^2p_4+x^3p_5.\\
		h_6:={g_6}_{\bullet}h & = & \zeta^2yp_1+\zeta y^2p_2+\zeta xp_3+\zeta^2 x^2p_4+x^3p_5. 
	\end{array}$$ 
	
	Recall from Lemma \ref{LEMMA} that we can define
	
	$$\begin{array}{lll}
		q_5(xy,x^3+y^3):= h_1+ \ h_2+ \ h_3+ \ h_4+ \ h_5+ \ h_6.\\ 
		q_4(xy,x^3+y^3):= h_1h_2+ \ h_1h_3+ \ h_1h_4+ \ \cdots + \  h_5h_6.\\
		q_3(xy,x^3+y^3):= h_1h_2h_3+ \ h_1h_2h_4+ \ \cdots + \ h_4h_5h_6.\\
		q_2(xy,x^3+y^3):= h_1h_2h_3h_4+ \ h_1h_2h_3h_5+\cdots + h_3h_4h_5h_6.\\
		q_1(xy,x^3+y^3):= h_1h_2h_3h_4h_5+ \ h_1h_2h_3h_4h_6+ \ \cdots + \ h_2h_3h_4h_5h_6.\\
		q_0(xy,x^3+y^3):= h_1h_2h_3h_4h_5h_6.
	\end{array}$$ 
	
	\begin{proposition}\label{dihedralprop} Let $f:(\mathbb{C}^2,0) \longrightarrow (\mathbb{C}^3,0)$ be a $6$-dihedral map germ and write it in the form
		
		\begin{equation*}
			f(x,y)=(x^3+y^3, \ xy, \  xp_1+x^2p_2+yp_3+y^2p_4+y^3p_5).
		\end{equation*}
		
		\noindent  where $p_i=p_i(x^3+y^3, \ xy)$. Then
		\begin{itemize}
			\item[$(a)$] The presentation matrix of $f_*{\cal O}_2$ as an ${\cal O}_{3} $-module via $f$ is
			
			$$\lambda[X,Y,Z]=\begin{bmatrix}
				-Z & p_1 & p_2 & p_3 & p_4 & p_5 \\
				p_2X+p_3Y & -Z & p_1 & p_4Y & p_5Y & -p_2 \\
				p_1X+p_4Y^2 & p_2X+p_3Y & -Z & p_5Y^2 & -p_2Y & -p_1 \\
				p_1Y & p_2Y & -p_5Y & p_5X-Z & p_3 & p_4 \\
				p_2Y^2 & -p_5Y^2 & -p_4Y & p_4X+p_1Y & p_5X-Z & p_3 \\
				-p_5Y^3 & -p_4Y^2 & -p_3Y & p_3X+p_2Y^2 & p_4X+p_1Y & p_5X-Z \
			\end{bmatrix},$$
			
			\noindent where $p_i=p_i(X,Y)$.
			
			\item[$(b)$] A defining equation of the image of $f$ is
			
			\begin{equation*}
				F(X,Y,Z)= Z^6-Q_{5}Z^{5}+Q_{4}Z^{4}-Q_3Z^3+Q_2Z^2-Q_1Z+Q_0.
			\end{equation*}
			
			where $Q_{6-k}(X,Y):=w^{\ast}(q_{6-k}(x^3+y^3,xy))$, i.e, we simply change the ``variables'' $x^3+y^3$ and $xy$ by the (target) variable $X$ and $Y$ in $q_{6-k}$ described just above this proposition.

			\item[$(c)$] A defining equation of the double point curve $D(f)$ of $f$ is
			
			\begin{equation*}
				D(f)=\textbf{V}\left(\eta_1\eta_2\eta_3\right)
			\end{equation*}
			where 
			$$\begin{array}{rcl} 
				\eta_1 & = & -p_1+p_3-p_2x+p_4x-p_2y+p_4y+p_5x^2+p_5xy+p_5y^2. \\  
				\eta_2 & = & p_1^2+p_1p_3+p_3^2+2p_1p_2x+p_2p_3x+p_1p_4x-p_3p_4x-p_1p_2y+p_2p_3y+p_1p_4y \\
				& & +2p_3p_4y+p_2^2x^2+p_2p_4x^2+p_4^2x^2-2p_1p_5x^2-
				p_3p_5x^2-p_2^2xy-p_2p_4xy-p_4^2xy \\
				& & +p_1p_5xy-p_3p_5xy+p_2^2y^2+p_2p_4y^2+p_4^2y^2+p_1p_5y^2+2p_3p_5y^2-2p_2p_5x^3-
				p_4p_5x^3 \\
				& & +2p_2p_5x^2y+p_4p_5x^2y-p_2p_5xy^2-2p_4p_5xy^2+p_2p_5y^3+2p_4p_5y^3+{p_5}^2x^4 \\
				& & -{p_5}^2x^3y-{p_5}^2xy^3+p_5^2y^4. \\
				\eta_3 & = & p_1^2x^2+p_1p_3xy+p_3^2y^2+p_1p_2x^3+2p_2p_3x^2y+2p_1p_4xy^2+p_3p_4y^3+p_2^2x^4+p_2p_4x^2y^2 \\ & & +p_4^2y^4.
			\end{array}$$
			
			\item[$(d)$] The multiplicity of the image of $f$ satisfies
			
			\begin{center}
				$2 \leq m(f(\mathbb{C}^2,0)) \leq 6$.
			\end{center}
		\end{itemize}
	\end{proposition}
	
	\begin{proof}$(a)$ Consider the basis $r_1=1$, $r_2=x$, $r_3=x^2$, $r_4=y$, $r_5=y^2$ and $r_6=y^3$ for $\mathcal{O}_2/(x^3+y^3,xy)$. Using the same notation as in Remark \ref{obs} we obtain that
		
		$$E=\begin{bmatrix}
			1 & 1 & 1 & 1 & 1 & 1 \\
			x & y & \zeta x & \zeta^2 x & \zeta y & \zeta^2y \\
			x^2 & y^2 & \zeta^2x^2 & \zeta x^2 & \zeta^2y^2 & \zeta y^2 \\
			y & x & \zeta^2 y & \zeta y & \zeta^2x & \zeta x \\
			y^2 & x^2 & \zeta y^2 & \zeta^2y^2 & \zeta x^2 & \zeta^2x^2 \\
			y^3 & x^3 & y^3 & y^3 & x^3 & x^3 \
		\end{bmatrix} \hspace{0.3cm} \textrm{and} \hspace{0.3cm} A_Z=\begin{bmatrix}
			h_1-Z & 0 & 0 & 0 & 0 & 0 \\
			0 & h_2-Z & 0 & 0 & 0 & 0 \\
			0 & 0 & h_3-Z & 0 & 0 & 0 \\
			0 & 0 & 0 & h_4-Z & 0 & 0 \\
			0 & 0 & 0 & 0 & h_5-Z & 0 \\
			0 & 0 & 0 & 0 & 0 & h_6-Z \
		\end{bmatrix}.$$
		
		The reflections of $D_6$ are $g_2,g_5$ and $g_6$ and their respective reflecting hyperplanes are $L_{H_1}=x-y$, $L_{H_2}=x-\zeta y$ and $L_{H_3}=x-\zeta^2 y$. Each reflection of $D_6$ has order $2$. By Proposition \ref{det} we obtain that $det(E)=27(x^3-y^3)^3$. The adjoint matrix of $E$ is
		
		$$adj(E)=9(x^3-y^3)^2\begin{bmatrix}
			x^3 & x^2 & x & -y^2 & -y & -1 \\
			-y^3 & -y^2 & -y & x^2 & x & 1 \\
			x^3 & \zeta^2x^2 & \zeta x & -\zeta y^2  & -\zeta^2y & -1 \\
			x^3 & \zeta x^2 & \zeta^2 x & -\zeta^2 y^2 & -\zeta y & -1 \\
			-y^3 & -\zeta^2 y^2 & -\zeta y & \zeta x^2 & \zeta^2 x & 1 \\
			-y^3 & -\zeta y^2 & -\zeta^2 y & \zeta^2 x^2 & \zeta x & 1 \
		\end{bmatrix}.$$
		
		Applying Corollary \ref{matrix2}, we obtain that a presentation matrix of $f_*({\cal O}_2)$ as a ${\cal O}_3$-module via $f$ is given by
		
		$$\lambda[X,Y,Z]=\begin{bmatrix}
			-Z & p_1 & p_2 & p_3 & p_4 & p_5 \\
			p_2X+p_3Y & -Z & p_1 & p_4Y & p_5Y & -p_2 \\
			p_1X+p_4Y^2 & p_2X+p_3Y & -Z & p_5Y^2 & -p_2Y & -p_1 \\
			p_1Y & p_2Y & -p_5Y & p_5X-Z & p_3 & p_4 \\
			p_2Y^2 & -p_5Y^2 & -p_4Y & p_4X+p_1Y & p_5X-Z & p_3 \\
			-p_5Y^3 & -p_4Y^2 & -p_3Y & p_3X+p_2Y^2 & p_4X+p_1Y & p_5X-Z \
		\end{bmatrix}$$
		as desired. \\

	 $(b)$ It follows by Theorem \ref{teeeo}.\\ 
		
		To prove $(c)$, we can apply Proposition \ref{princi} to obtain
		
		\begin{equation*}
			D(f)=\textbf{V}\left( \dfrac{(h-{g_2}_{\bullet}h)(h-{g_3}_{\bullet}h)(h-{g_4}_{\bullet}h)(h-{g_5}_{\bullet}h)(h-{g_6}_{\bullet}h)}{3(x-y)(x-\zeta y)(x-\zeta^2y)}\right).
		\end{equation*}
		
		Note that 
		
		$$\begin{array}{rcl} 
			h-{g_2}_{\bullet}h & = & (y-x)(-p_1+p_3+p_2(x+y)+p_4(x+y)-p_5(x^2+xy+y^2)). \\  
			h-{g_3}_{\bullet}h & = & p_1(x-\zeta^2x)+p_2(x^2-\zeta x^2)+p_3(y-\zeta y)+p_4(y^2-\zeta^2y^2). \\
			h-{g_4}_{\bullet}h & = & p_1(x-\zeta x)+p_2(x^2-\zeta^2x^2)+p_3(y-\zeta^2y)+p_4(y^2-\zeta y^2). \\
			h-{g_5}_{\bullet}h & = & p_1(x-\zeta y)+p_2(x^2-\zeta^2y^2)+p_3(y-\zeta^2x)+p_4(y^2-\zeta x^2)+p_5(y^3-x^3).\\
			h-{g_6}_{\bullet}h & = & p_1(x-\zeta^2y)+p_2(x^2-\zeta y^2)+p_3(y-\zeta x)+p_4(y^2-\zeta^2x^2)+p_5(y^3-x^3). 
		\end{array}$$
		
		Therefore, a calculation shows that $D(f)=\textbf{V}\left(f_1f_2f_3\right)$.\\ 
		
		The proof of $(d)$ follows by Theorem \ref{teo5.2} and Remark \ref{cormult}.\end{proof}
	
	\begin{remark} A straightforward (but tedious) calculation can be done to present the $Q_{i}'s$ coefficients in Proposition \ref{dihedralprop}(b) explicitly. For instance, 
		
		\begin{center}
			$q_5=h_1+h_2+h_3+h_4+h_5+h_6=3p_5(x^3+y^3)$, and\\
			$q_4=-6xyp_1p_3-3p_1p_2(x^3+y^3)-3p_3p_4(x^3+y^3)-6x^2y^2p_2p_4+3p_5^2(x^3+y^3)^2+3p_5^2x^3y^3$
		\end{center}

		\noindent therefore $Q_5(X,Y)=3Xp_5$ and $Q_4(X,Y)=-6Yp_1p_3-3Xp_1p_2-3Xp_3p_4-6Y^2p_2p_4+3X^2p_5^2+3Y^3p_5^2$. In the same way, we obtain from $q_3,q_2,q_1$ and $q_0$ that\\
		
		$Q_3(X,Y)=-X(-p_1^3-p_3^3-Xp_2^3-Xp_4^3+9Xp_1p_2p_5+3Xp_3p_4p_5-X^2p_5^3)-XY(-3p_2^2p_3-3p_1p_4^2+12p_1p_3p_5+12Y^2p_2p_4p_5-6Y^2p_5^3)-Y^2(-6p_2p_3^2-6p_1^2p_4+2Yp_2^3+2Yp_4^3-12Yp_1p_2p_5+12Yp_3p_4p_5)$. \\
		
		$Q_2(X,Y)=9XYp_1^2p_2p_3+9Y^2p_1^2p_3^2+9X^2p_1p_2p_3p_4+9XYp_1p_3^2p_4+3X^2p_1^3p_5+9Y^3p_1^2p_2^2+9XY^2p_1p_2^2p_4+9XY^2p_2p_3p_4^2+9Y^3p_3^2p_4^2-6Y^3p_1^3p_5+3X^3p_2^3p_5+9X^2Yp_2^2p_3p_5+9XY^2p_2p_3^2p_5+6Y^3p_3^3p_5+9XY^2p_1^2p_4p_5-9X^3p_1p_2p_5^2-9X^2Yp_1p_3p_5^2+9Y^4p_2^2p_4^2-9XY^3p_2^3p_5-18Y^4p_2^2p_3p_5+18Y^4p_1p_4^2p_5+3XY^3p_4^3p_5+18XY^3p_1p_2p_5^2 \linebreak -9X^2Y^2p_2p_4p_5^2-18XY^3p_3p_4p_5^2+3X^2Y^3p_5^4+3Y^6p_5^4 $. \\
		
		%$Q_1(X,Y)=-[3XYp_1^4p_3+3X^2p_1p_2p_3^3+3XYp_1p_3^4+3X^2p_1^3p_3p_4+6Y^3p_1^4p_2+3X^2Yp_1p_2^3p_3+18XY^2p_1p_2^2p_3^2+12Y^3p_1p_2p_3^3+12XY^2p_1^3p_2p_4+12Y^3p_1^3p_3p_4+3X^3p_2^3p_3p_4+9X^2Yp_2^2p_3^2p_4+12XY^2p_2p_3^3p_4+6Y^3p_3^4p_4+9X^2Yp_1^2p_2p_4^2+18XY^2p_1^2p_3p_4^2+3X^3p_1p_2p_4^3+3X^2Yp_1p_3p_4^3-9X^2Yp_1^2p_2p_3p_5-9XY^2p_1^2p_3^2p_5-9X^3p_1p_2p_3p_4p_5\linebreak -9X^2Yp_1p_3^2p_4p_5-3X^3p_1^3p_5^2+3XY^3p_1p_2^4+12Y^4p_1p_2^3p_3+3X^2Y^2p_2^4p_4+12Y^4p_1p_3p_4^3+3X^2Y^2p_2p_4^4+3XY^3p_3p_4^4-9XY^3p_1^2p_2^2p_5-9X^2Y^2p_1p_2^2p_4p_5-9X^2Y^2p_2p_3p_4^2p_5-9XY^3p_3^2p_4^2p_5+9XY^3p_1^3p_5^2-3X^4p_2^3p_5^2-9X^3Yp_2^2p_3p_5^2-9X^2Y^2p_2p_3^2p_5^2-3XY^3p_3^3p_5^2-9X^2Y^2p_1^2p_4p_5^2+3X^4p_1p_2p_5^3+3X^3Yp_1p_3p_5^3-6Y^5p_2^4p_4-6Y^5p_2p_4^4-9XY^4p_2^2p_4^2p_5+12X^2Y^3p_2^3p_5^2+27XY^4p_2^2p_3p_5^2-9XY^4p_1p_4^2p_5^2-3X^2Y^3p_1p_2p_5^3+3X^3Y^2p_2p_4p_5^3+3X^2Y^3p_3p_4p_5^3-6Y^6p_2^3p_5^2-6Y^6p_4^3p_5^2-12Y^6p_1p_2p_5^3+12Y^6p_3p_4p_5^3-3XY^6p_5^5+18Y^5p_2p_3^2p_5^2+18Y^5p_1^2p_4p_5^2$. \\
		
		$Q_1(X,Y)=-[3XYp_1^4p_3+3X^2p_1p_2p_3^3+3XYp_1p_3^4+3X^2p_1^3p_3p_4+6Y^3p_1^4p_2+3X^2Yp_1p_2^3p_3+18XY^2p_1p_2^2p_3^2+12Y^3p_1p_2p_3^3+12XY^2p_1^3p_2p_4+12Y^3p_1^3p_3p_4+3X^3p_2^3p_3p_4+9X^2Yp_2^2p_3^2p_4+12XY^2p_2p_3^3p_4+6Y^3p_3^4p_4+9X^2Yp_1^2p_2p_4^2+18XY^2p_1^2p_3p_4^2+3X^3p_1p_2p_4^3+3X^2Yp_1p_3p_4^3-9X^2Yp_1^2p_2p_3p_5-9XY^2p_1^2p_3^2p_5-9X^3p_1p_2p_3p_4p_5\linebreak -9X^2Yp_1p_3^2p_4p_5-3X^3p_1^3p_5^2+3XY^3p_1p_2^4+12Y^4p_1p_2^3p_3+3X^2Y^2p_2^4p_4+12Y^4p_1p_3p_4^3+3X^2Y^2p_2p_4^4+3XY^3p_3p_4^4-9XY^3p_1^2p_2^2p_5-9X^2Y^2p_1p_2^2p_4p_5-9X^2Y^2p_2p_3p_4^2p_5-9XY^3p_3^2p_4^2p_5+9XY^3p_1^3p_5^2-3X^4p_2^3p_5^2-9X^3Yp_2^2p_3p_5^2-9X^2Y^2p_2p_3^2p_5^2-3XY^3p_3^3p_5^2-9X^2Y^2p_1^2p_4p_5^2+3X^4p_1p_2p_5^3+3X^3Yp_1p_3p_5^3-6Y^5p_2^4p_4-6Y^5p_2p_4^4-9XY^4p_2^2p_4^2p_5+12X^2Y^3p_2^3p_5^2+27XY^4p_2^2p_3p_5^2-9XY^4p_1p_4^2p_5^2-3X^2Y^3p_1p_2p_5^3+3X^3Y^2p_2p_4p_5^3+3X^2Y^3p_3p_4p_5^3-6Y^6p_2^3p_5^2-6Y^6p_4^3p_5^2-12Y^6p_1p_2p_5^3+12Y^6p_3p_4p_5^3-3XY^6p_5^5+18Y^5p_2p_3^2p_5^2+18Y^5p_1^2p_4p_5^2]$. \\
		
		\noindent and
		
		\begin{flushleft}
			$Q_0(X,Y)=X^2p_1^3p_3^3+Y^3p_1^6+3XY^2p_1^3p_2p_3^2-2Y^3p_1^3p_3^3+X^3p_2^3p_3^3+3X^2Yp_2^2p_3^4+3XY^2p_2p_3^5+Y^3p_3^6+3XY^2p_1^5p_4+3XY^2p_1^2p_3^3p_4+3X^2Yp_1^4p_4^2+X^3p_1^3p_4^3-3X^2Yp_1^4p_3p_5-3X^3p_1^3p_3p_4p_5+XY^3p_1^3p_2^3+6Y^4p_1^3p_2^2p_3+3X^2Y^2p_2^4p_3^2+6XY^3p_2^3p_3^3+3Y^4p_2^2p_3^4+3X^2Y^2p_1^2p_2^3p_4+9XY^3p_1^2p_2^2p_3p_4+18Y^4p_1^2p_2p_3^2p_4+3Y^4p_1^4p_4^2+3X^3Yp_1p_2^3p_4^2+9X^2Y^2p_1p_2^2p_3p_4^2+9XY^3p_1p_2p_3^2p_4^2+6Y^4p_1p_3^3p_4^2+6XY^3p_1^3p_4^3+X^4p_2^3p_4^3+3X^3Yp_2^2p_3p_4^3+3X^2Y^2p_2p_3^2p_4^3+XY^3p_3^3p_4^3+3X^2Y^2p_1^2p_4^4-3XY^3p_1^4p_2p_5+6Y^4p_1^4p_3p_5-3X^3Yp_1p_2^3p_3p_5-9X^2Y^2p_1p_2^2p_3^2p_5-12XY^3p_1p_2p_3^3p_5-6Y^4p_1p_3^4p_5-3X^2Y^2p_1^3p_2p_4p_5-3X^4p_2^3p_3p_4p_5-9X^3Yp_2^2p_3^2p_4p_5-9X^2Y^2p_2p_3^3p_4p_5-3XY^3p_3^4p_4p_5-9X^2Y^2p_1^2p_3p_4^2p_5+X^4p_1^3p_5^3+3XY^4p_2^5p_3+3Y^5p_2^4p_3^2-6Y^5p_1^2p_2^3p_4-9XY^4p_1p_2^3p_4^2 -18Y^5p_1p_2^2p_3p_4^2-4X^2Y^3p_2^3p_4^3-9XY^4p_2^2p_3p_4^3-6Y^5p_2p_3^2p_4^3+3Y^5p_1^2p_4^4+3XY^4p_1p_4^5-3X^2Y^3p_1p_2^4p_5-12Y^5p_1^3p_2p_4p_5-3X^3Y^2p_2^4p_4p_5+3X^2Y^3p_2^3p_3p_4p_5+18XY^4p_2^2p_3^2p_4p_5+12Y^5p_2p_3^3p_4p_5-18XY^4p_1^2p_2p_4^2p_5-3X^2Y^3p_1p_2p_4^3p_5-12XY^4p_1p_3p_4^3p_5+9XY^4p_1^2p_2p_3p_5^2+9Y^5p_1^2p_3^2p_5^2+9X^2Y^3p_1p_2p_3p_4p_5^2+9XY^4p_1p_3^2p_4p_5^2-4X^2Y^3p_1^3p_5^3+X^5p_2^3p_5^3+3X^4Yp_2^2p_3p_5^3+3X^3Y^2p_2p_3^2p_5^3+X^2Y^3p_3^3p_5^3+3X^3Y^2p_1^2p_4p_5^3+Y^6p_2^6+2Y^6p_2^3p_4^3+Y^6p_4^6+6Y^6p_1p_2^4p_5+9XY^5p_2^4p_4p_5+12Y^6p_2^3p_3p_4p_5-12Y^6p_1p_2p_4^3p_5-3XY^5p_2p_4^4p_5-6Y^6p_3p_4^4p_5+9Y^6p_1^2p_2^2p_5^2+9XY^5p_1p_2^2p_4p_5^2+9XY^5p_2p_3p_4^2p_5^2+9Y^6p_3^2p_4^2p_5^2+2Y^6p_1^3p_5^3-5X^3Y^3p_2^3p_5^3-12X^2Y^4p_2^2p_3p_5^3-9XY^5p_2p_3^2p_5^3-2Y^6p_3^3p_5^3-9XY^5p_1^2p_4p_5^3+3X^2Y^4p_1p_4^2p_5^3-3X^3Y^3p_1p_2p_5^4-3X^2Y^4p_1p_3p_5^4+9Y^7p_2^2p_4^2p_5^2+5XY^6p_2^3p_5^3+6Y^7p_2^2p_3p_5^3-6Y^7p_1p_4^2p_5^3+XY^6p_4^3p_5^3+9XY^6p_1p_2p_5^4+6Y^7p_1p_3p_5^4-3X^2Y^5p_2p_4p_5^4 -3XY^6p_3p_4p_5^4+6Y^8p_2p_4p_5^4+Y^9p_5^6$.
		\end{flushleft}
		
	\end{remark}
	
	\subsection{Quasihomogeneous reflection maps}
	
	$ \ \ \ \ $ Another application of our results is about quasihomogeneous map germs. A polynomial $p(x_1,\cdots,x_n)$ is \textit{quasihomogeneous} if there are positive coprime integers $b_1,\cdots,b_n$, and an integer $d$ such that
	 
\begin{center}
	 $p(k^{b_1}x_1,\cdots,k^{b_n}x_x)=k^dp(x_1,\cdots,x_n)$. 
\end{center}
	 
	 The number $b_i$ is called the weight of the variable $x_i$ and $d$ is called the weighted degree of $p$. In this case, we say $p$ is of type $(d; b_1,\cdots,b_n)$. This definition extends to polynomial map germs $f:(\mathbb{C}^n,0)\rightarrow (\mathbb{C}^p,0)$ by just requiring each coordinate function $f_i$ to be quasihomogeneous of type $(d_i; b_1,\cdots,b_n)$, for fixed weights $b_1,\cdots,b_n$. In particular, for a quasihomogeneous map germ $f:(\mathbb{C}^2,0)\rightarrow (\mathbb{C}^3,0)$ we say that it is quasihomogeneous of type $(d_1,d_2,d_3; b_1,b_2)$.\\
	
	Marar and Nuño-Ballesteros studied in \cite{[3]} the case where $f:(\mathbb{C}^2,0) \longrightarrow (\mathbb{C}^3,0)$ is a corank $2$ quasihomogeneous and finitely determined map germ. They say that the mere existence of this kind of maps is quite a surprise. Indeed, the three adjectives create tremendous restrictions and examples seem hard to find. They showed that if $f=(x^2,y^2,h(x,y))$ is finitely determined, then $f$ is in fact homogeneous, i.e. $b_1=b_2=1$ (see \cite[Th. 3.4]{[3]}). In particular, there is no finitely determined quasihomogeneous double fold map germ with distinct weights. On the other hand, examples of finitely determined homogeneous (where $b_1=b_2=1$) reflected graph map germs exist, see for instance \cite[Example 3.6]{[3]} and \cite[Example 16]{[5]}. Thus, we can consider the following question:\\ 
	
	\noindent \textbf{Question:} \textit{Is there any corank $2$ reflected graph map germ $f=(w_1,w_2,h)$ from $(\mathbb{C}^2,0)$ to $(\mathbb{C}^3,0)$ such that $f$ is finitely determined and quasihomogeneous with distinct weights?}\\
	
	Since the coordinate functions $w_1$ and $w_2$ to the orbit map $w=(w_1,w_2)$ of $G$ are always homogeneous (see \cite[Ch. 9]{[2]}), we will restrict ourselves to studying this question only for the group $\mathbb{Z}_r \times \mathbb{Z}_s$, where the orbit map $w=(x^r,x^s)$ of $\mathbb{Z}_r \times \mathbb{Z}_s$ is a quasihomogeneous map. If $f(x,y)=(x^r,y^s,h(x,y))$ then the corank $2$ hypothesis implies that $r,s\geq 2$ and $h \in \textbf{m}^2$, where $\textbf{m}$ denotes the maximal ideal of $\mathcal{O}_2$. The following lemma can be viewed as an extension of \cite[Th. 3.4]{[3]} for the reflection group $\mathbb{Z}_r \times \mathbb{Z}_s$. We will consider $\mathbb{Z}_r \times \mathbb{Z}_s$ as a subgroup of $GL_2(\mathbb{C})$ generated by the reflections

	$$ R^{'}= \begin{bmatrix}
		\theta & 0  \\
		0 & 1 \\
	\end{bmatrix}, \ \ \ \ \ S^{'} = \begin{bmatrix} 1 & 0  \\
		0 & \xi \\
	\end{bmatrix}.$$
	
	\noindent where $\theta$ (respectively $\xi$) is a primitive $\textbf{r}$-th (respectively, $\textbf{s}$-th) root of unity.\\
	
	Denote the elements of $\mathbb{Z}_r \times \mathbb{Z}_s$ by $g_{i,j}$ where $i \in \lbrace 1,\cdots, r \rbrace$ and $j \in \lbrace 1,\cdots, s \rbrace$, with $g_{r,s}=Id$, the identity matrix. Note that after an eventual reordering of the indices $i,j$ we have that ${g_{i,j}}_{\bullet}h(x,y)=h(\theta^i x,\xi^j y)$.
	
	\begin{lemma}\label{homo} Let $f:(\mathbb{C}^2,0) \longrightarrow (\mathbb{C}^3,0)$ be a reflected graph map germ given by $f(x,y)=(x^r,y^s,h(x,y))$ with $r,s\geq 2$ and $h \in \textbf{m}^2$. If $f$ is quasihomogeneous and finitely determined then $f$ is homogeneous.
	\end{lemma}
	
	\begin{proof} The case where $r,s=2$ was considered in \cite[Th. 3.4]{[3]}. Therefore, we can suppose that $(r,s) \neq (2,2)$. Denote the weight of $x$ by $a$ and the weight of $y$ by $b$. By hypothesis we have that $f$ is quasihomogeneous, therefore we can write $h$ in the form
		\begin{equation}\label{eq4}
			h(x,y)=x^{\alpha}y^{\beta}(c_k(x^b)^k+c_{k-1}(x^b)^{k-1}y^a+\cdots+c_1x^b(y^a)^{k-1}+c_0(y^a)^k)
		\end{equation}
		
		\noindent for some non-negative integers $\alpha,\beta,k$, where $c_0,\cdots,c_k \in \mathbb{C}$ and $c_0,c_k \neq 0$. We will show that $\alpha,\beta = 0$. Since $\det(jac(w))=rsx^{r-1}y^{s-1}$ then by Proposition \ref{princi} we obtain that
		\begin{equation}\label{23}
			{D}(f)=\textbf{V}\left( \displaystyle \dfrac{ \prod_{(i,j)\neq (r,s)}\left(h-{g_{i,j}}_{\bullet}h\right) }{x^{r-1}y^{s-1}}\right).
		\end{equation}
		
		Suppose that $\beta \neq 0$, then the restriction of $f$ to $\textbf{V}(y)$ is $r$-to-$1$ (see \cite[Lemma 6.1]{[6b]}). Since $f$ is finitely determined, it follows that $1\leq r \leq 2$. If we suppose that $\alpha \neq 0$, we obtain with a similar argument that $1 \leq s \leq 2$. Hence, if $r,s\geq 3$ then $\alpha=\beta =0$. Let us consider the following remaining cases:\\

		\noindent\textbf{Case a.1:} $r=2$ and $s\geq 3$.\\
		
		Since $s\geq 3$ we have that $\alpha=0$. Suppose that $\beta\geq 1$, then by (\ref{eq4}) and (\ref{23}) we obtain that
		
		\begin{center}
			${D}(f)=\textbf{V}(y^{\beta(2s-1)-(s-1)}\lambda_1(x,y))$
		\end{center}
		
		\noindent for some $\lambda_1(x,y)$ in $\mathcal{O}_2$. Now, note that $\beta(2s-1)-(s-1)=\beta s + \beta(s-1)-(s-1)\geq 3$. Therefore, $D(f)$ is not reduced, thus it follows by \cite[Cor. 3.5]{[11]} that $f$ is not finitely determined, a contradiction. Hence, we obtain that $\beta=0$. \\
		
		\noindent\textbf{Case a.2:} $s=2$ and $r\geq 3$. The proof of this case is similar to the one given to show Case a.1.\\
		
		Finally, we will show that the weights of $f$ are equal, i.e., $a=b=1$. Since $\alpha=\beta=0$ then by (\ref{eq4}) and (\ref{23}) we obtain that
		
		\begin{center}
			$D(f)=\textbf{V}(x^{b(r-1)-(r-1)}y^{a(s-1)-(s-1)}\lambda_2(x,y))$
		\end{center}
		
		\noindent for some $\lambda_2(x,y)$ in $\mathcal{O}_2$. Therefore, since $f$ is finitely determined it follows that 
		
		\begin{equation}\label{eq5}
			0\leq (r-1)(b-1)\leq 1 \ \ \ and \ \ \ 0 \leq (s-1)(a-1)\leq 1.
		\end{equation}
		
		From (\ref{eq5}), we obtain that if $r\geq 3$ then $b=1$. On the other hand, if $s\geq 3$ then $a=1$. Hence, if $r,s\geq 3$ then $a=b=1$. Let us consider the following remaining cases:\\

		\noindent\textbf{Case b.1:} $r=2$ and $s\geq 3$.\\
		
		Note that in this case $f=(f_1,f_2,f_3)=(x^2,y^s,c_k(x^b)^k+\cdots+c_0(y^a)^k)$ with $c_0,c_k \neq 0$. From (\ref{eq5}) we obtain that $a=1$ and $1\leq b \leq 2$. Suppose that $b=2$, then the weighted degree of $f_1,f_2$ and $f_3$ are $2,2s$ and $2k$. By \cite[Prop. 1.4]{mondformulas} we obtain that the number $C(f)$ of cross-caps of $f$ is given by
		
		\begin{equation}\label{eq6}
			C(f)=\dfrac{1}{2}\left((2s-1)(2k-2)+(2s-1) \right)
		\end{equation}
		
		\noindent which is not an integer number, since the numerator of (\ref{eq6}) is an odd integer number. In particular, this implies that $f$ is not finitely determined, a contradiction. Therefore, we obtain that $a=b=1$.\\
		
		\noindent\textbf{Case b.2:} $s=2$ and $r\geq 3$. The proof of this case is similar to the one given to show Case b.1.\end{proof}\\
	
	We note that for the group $\mathbb{Z}_1 \times \mathbb{Z}_d$, where $\mathbb{Z}_1$ denotes the trivial group, it is not hard to find quasihogeneous finitely determined map germs which are not homogeneous. For instance, $f(x,y)=(x,y^4,y^6+xy)$ is an example of a corank $1$ finitely determined map germ which is quasihomogeneous of type $(5,4,6; 1,5)$.
	
	\begin{remark}  All figures used in this work were created by the authors using the software Surfer \rm\cite{surfer}.
	\end{remark}
	
\section*{Acknowlegments}

$ \ \ \ \ $ This work constitutes a part of the first author's Ph.D. thesis at the Federal University of Paraíba under the supervision of Otoniel Nogueira da Silva to whom he would like to express his deepest gratitude. We would like to express our sincere gratitude to the referees for their valuable comments and suggestions, which have significantly improved the quality and clarity of this work. The authors would like to thank Juan José Nuño-Ballesteros and Guillermo Peñafort-Sanchis for many valuable comments on this work. Milena Barbosa Gama acknowledges support by CAPES (Coordenação de Aperfeiçoamento de Pessoal de Nível Superior). Otoniel Nogueira da Silva acknowledges support by CNPq (Conselho Nacional de Desenvolvimento Científico e Tecnológico) grant Universal 407454/2023-3.

%\section*{Conflict of interest statement}

%$ \ \ \ \ $ On behalf of all authors, the corresponding author states that there is no conflict of interest.

	\begin{flushleft}
		$\bullet$ Gama, M.B.\\
		\textit{milena.gama@academico.ufpb.br}\\
		Universidade Federal da Paraíba, 58.051-900, João Pessoa, PB, Brazil.\\
		
		$ \ \ $\\
		
		$\bullet$ Silva, O.N.\\
		\textit{otoniel.silva@academico.ufpb.br}\\
		Universidade Federal da Paraíba, 58.051-900, João Pessoa, PB, Brazil.	
	\end{flushleft}

%\begin{flushleft}
%	\textbf{Data availability statement}
%	\end{flushleft}	
	
%\begin{flushleft}
	
%	Data available with the article or supplementary information:\\

%The authors declare that the data supporting the findings of this study are available within the article itself and in the cited bibliography. Access information for the cited bibliography is described in the bibliography itself. There is no supplementary data that requires a specific link or repository to be cited in this section.
%	\end{flushleft}	
	
\end{document}